\documentclass[10pt]{amsart}
\usepackage{latexsym,amsfonts,amsmath,graphics,enumerate,amsthm,amssymb,color}
\newcommand\closure{\operatorname{cl}}
\newcommand\interior{\operatorname{int}}
\newcommand\intersection{\cap}
\newcommand\grad{\nabla}

\newcommand\detourmet{\delta}
\newcommand\hil{d_H}
\newcommand\funk{d_F}
\newcommand\rfunk{d_R}
\newcommand\thomp{d_T}
\newcommand\after{\circ}
\newcommand\union{\cup}
\newcommand\R{\mathbb{R}}
\newcommand\C{\mathbb{C}}

\newcommand{\Rplus}{\mathbb{R}_+}

\newcommand\N{\mathbb{N}}
\newcommand\coll{\operatorname{Coll}}
\newcommand\isom{\operatorname{Isom}}
\newcommand\aut{\operatorname{Aut}}

\newcommand\gaugebigbracket[3]{M_{#1}\big({#2},{#3}\big)}
\newcommand\gauge[3]{M_{#1}({#2},{#3})}
\newcommand\inver{*}

\newcommand\barphi{\phi}
\newcommand\relint{\operatorname{rel}\operatorname{int}}
\newcommand\dotprod[2]{\langle{#1},{#2}\rangle}
\newcommand\linspace{V}

\newcommand\eucl{\mathcal{E}}
\newcommand\bilinear{\mathcal{B}}
\newcommand\mybilin{B}

\newcommand\dee{D}

\newtheorem*{mytheorem}{Theorem \ref{thm:thompson_isometries}}
\newtheorem{theorem}{Theorem}[section]
\newtheorem{Definition}[theorem]{Definition}
\newtheorem{lemma}[theorem]{Lemma}
\newtheorem{corollary}[theorem]{Corollary}
\newtheorem{proposition}[theorem]{Proposition}
\newtheorem{property}[theorem]{Property}
\newenvironment{definition}{\begin{Definition}\rm}{\end{Definition}}
\renewcommand{\theenumi}{\roman{enumi}}

\renewcommand{\theenumi}{\roman{enumi}}

\newcommand\proj{P}
\newcommand\directproduct{\oplus}

\title{Gauge-reversing maps on cones, and Hilbert and Thompson isometries}
\author{Cormac Walsh}
\address{INRIA \& CMAP, \'{E}cole Polytechnique,\\ 
91128 Palaiseau, France}
\begin{document}

\begin{abstract}
We show that a cone admits a gauge-reversing map if and only if
it is a symmetric cone.
We use this to prove that every isometry of a Hilbert geometry
is a collineation unless the Hilbert geometry is the projective space
of a non-Lorentzian symmetric cone,
in which case the collineation group is of index two in the isometry group.
We also determine the isometry group of the Thompson geometry on a cone.
\end{abstract}

\date{\today}

\maketitle

\section{Introduction}

Consider a proper open convex cone $C$ in a finite-dimensional vector space
$\linspace$. Associated to $C$, there is a natural partial order on $\linspace$
defined by $x\le_C y$ if $y-x\in\closure C$.
The \emph{gauge} on $C$ is defined by
\begin{align*}
\gauge{C}{x}{y} := \inf\{ \lambda>0 \mid x\le_C \lambda y\},
\qquad\text{for all $x,y\in C$.}
\end{align*}
Related to the gauge are the following two metrics on $C$.
\emph{Hilbert's projective metric} is defined to be
\begin{align*}
\hil(x,y) := \log\gauge{C}{x}{y}\gauge{C}{y}{x},
\qquad\text{for all $x,y\in C$.}
\end{align*}
This is actually a pseudo-metric since $\hil(x,\lambda x)=0$ for any $x\in C$
and $\lambda>0$. On the projective space of the cone it is a genuine metric.
\emph{Thompson's metric} is defined to be
\begin{align*}
\thomp(x,y) := \log\max\big(\gauge{C}{x}{y},\gauge{C}{y}{x}\big),
\qquad\text{for all $x,y\in C$.}
\end{align*}

In this paper, we study the maps between cones that preserve or reverse
the gauge, or are isometries of one of the two metrics.
Recall that a map $\phi\colon C\to C'$ between two proper open convex cones
is said to be gauge preserving if
$\gauge{C'}{\phi x}{\phi y} = \gauge{C}{x}{y}$
and gauge-reversing  if $\gauge{C'}{\phi x}{\phi y} = \gauge{C}{y}{x}$,
for all $x$ and $y$ in $C$.
Obviously, both types of map are isometries of the Hilbert and Thompson metrics.

Any linear isomorphism between $C$ and $C'$ is clearly gauge-preserving.
Noll and Sch\"affer~\cite{noll_schaffer_orders_gauge} showed that the converse
is also true: every gauge-preserving bijection between two
finite-dimensional cones is a linear isomorphism.

On a symmetric cone, that is, one that is homogeneous and self-dual,
Vinberg's \emph{$*$-map} is
gauge-reversing~\cite{kai_order_reversing}.
We show that gauge-reversing maps exist only on symmetric cones.

\begin{theorem}
\label{thm:symmetric}
Let $C$ be a proper open convex cone in a finite-dimensional real vector
space. Then, $C$ admits a gauge-reversing map if and only if $C$ is
symmetric.
\end{theorem}

Noll and Sch\"affer~\cite[p.~377]{noll_schaffer_orders_gauge} raise the
question of whether the existence of a gauge-reversing bijection between
two cones requires that they be linearly isomorphic.
We answer this question in the affirmative for finite-dimensional cones.

\begin{corollary}
\label{cor:two_cones}
If there is a gauge-reversing bijection between two finite-dimensional cones,
then the cones are linearly isomorphic.
\end{corollary}

\newcommand\pos{\operatorname{Pos}}

Let $D:=\proj(C)$ be the projective space of the cone $C$,
and consider the Hilbert metric on $D$.
The isometry group of this metric was first studied by
de la Harpe~\cite{delaharpe}. He showed that if the closure of $D$ is strictly
convex, then every isometry is a collineation, that is, arises as the
projective action of a linear map on the cone.
He also noted that there exist isometries that are not collineations
in the case of the positive cone (where $D$ is an open simplex) and in the case
of the cone of positive-definite symmetric matrices.
Both of these cones are symmetric.
De la Harpe asked, in general, when do the isometry group $\isom(D)$
and the collineation $\coll(D)$ group coincide?

Moln\'ar~\cite{molnar_thompson_isometries} determined the isometry group
of the Hilbert metric in the case of another symmetric
cone, the cone $\pos(\C,n)$; $n\ge 3$ of positive definite Hermitian matrices
with complex entries. One may interpret his results as saying that each
isometry is the projective action of either a gauge-preserving
or a gauge-reversing map on the cone.
Moln\'ar and Nagy~\cite{molnar_2d} extended this result to the case where
$n=2$, in which case, of course, the Hilbert geometry is isometric to
$3$-dimensional hyperbolic space.

These results were generalised to all finite-dimensional symmetric cones
by Bosch\'e~\cite{bosche}, using Jordan algebra techniques.

In a different direction, it was shown in~\cite{lemmens_walsh_polyhedral}
that for polyhedral Hilbert geometries every isometry is a collineation,
unless the domain $D$ is a simplex, in which case the collineation group has
index two in the isometry group.

We show the following.

\begin{theorem}
\label{thm:hilbert_isometries}
Let $(D,\hil)$ be a finite-dimensional Hilbert geometry, and let $C$ be a
cone over $D$.
Every isometry of $(D,\hil)$ arises as the projective action of either
a gauge-preserving or gauge-reversing map of $C$.
\end{theorem}

Combining this with Theorem~\ref{thm:symmetric}, gives us the isometry group
of any Hilbert metric.

\begin{corollary}
\label{cor:hilbert_isometries}
If $C$ is symmetric and not Lorentzian, then $\coll(D)$ is a normal subgroup
of index two in $\isom(D)$. Otherwise $\isom(D)=\coll(D)$.
\end{corollary}

This result had been conjectured in~\cite{lemmens_walsh_polyhedral}.
It also resolves some conjectures of de la Harpe, namely that $\isom(D)$
is a Lie group, and that $\isom(D)$ acts transitively on $D$ if and only if
$\coll(D)$ does,

We also determine the isometry group of the Thompson metric.

\begin{theorem}
\label{thm:thompson_isometries}
Let $C$ and $C'$ be proper open convex cones, and
let $\phi\colon C\to C'$ be a surjective isometry of the Thompson metric.
Then, there exist decompositions $C=C_1\directproduct C_2$ and
$C'=C'_1\directproduct C'_2$ such that $\phi$ takes the form
$\phi(x_1+x_2)= (\phi_1(x_1)+\phi_2(x_2))$, where $\phi_1$ is a
gauge-preserving map from $C_1$ to $C'_1$, and $\phi_2$ is a gauge-reversing
map from $C_2$ to $C'_2$.
\end{theorem}

The first to study the isometries of the Thompson metric were
Noll and Sch\"affer~\cite{noll_schaffer_orders_gauge}.
They showed that, in the case where the cone order
of either $C$ or $C'$ is \emph{loose}, every such isometry is either
gauge-preserving or gauge-reversing. Here loose means that for all $x$ and
$y$ in the cone, the set $\{x,y\}$ has neither an infimum nor a supremum
unless $x$ and $y$ are comparable. In particular, they showed that both
the Lorentz cone and the cone positive definite symmetric matrices
are loose.

The isometry group of the Thompson metric has been worked out
by Moln\'ar~\cite{molnar_thompson_isometries} in the case of
the cone of positive-definite complex Hermitian matrices,
and by Bosch\'e~\cite{bosche} for general symmetric cones.

The plan of the paper is as follows. We recall some background material
in Section~\ref{sec:preliminaries}. We then prove the homogeneity of
any cone admitting a gauge-reversing map in Section~\ref{sec:homogeneity}.
An important tool we will use in much of the paper is the
\emph{horofunction boundary};
we recall its definition in Section~\ref{sec:horoboundary} and describe known
results about it in the case of the Hilbert geometry in
Section~\ref{sec:hilbert_horoboundary}. Using these results, we finish the
proof of Theorem~\ref{thm:symmetric} and of Corollary~\ref{cor:two_cones}
in Section~\ref{sec:self_duality}.
In Section~\ref{sec:hilbert_isometries}, we study the isometries of the
Hilbert geometry and prove Theorem~\ref{thm:hilbert_isometries}
and~Corollary~\ref{cor:hilbert_isometries}.
Sections~\ref{sec:product} and~\ref{sec:horofunction_thompson} are devoted to
the study of the horofunction boundaries of, respectively, product spaces
and Thompson geometries. These results are then used
in Section~\ref{sec:thompson_isometries} to prove
Theorem~\ref{thm:thompson_isometries}.

\emph{Acknowledgements.}
I greatly benefited from many discussions with Bas Lemmens concerning this
work. This work was partially supported by the ANR `Finsler'.

\section{Preliminaries}
\label{sec:preliminaries}

\subsection{Gauge-preserving and gauge-reversing maps}

Let $C$ be an open convex cone in a real finite-dimensional vector space
$\linspace$. In other words, $C$ is an open convex set that is invariant under
multiplication by positive scalars. We use $\closure$ to denote
the closure of a set. If $\closure{C}\cap (-\closure{C})=\{0\}$,
then $C$ is called a \emph{proper} open convex cone.

As described in the introduction, $C$ induces a natural partial order $\le_C$
on $\linspace$, and this is used to define the gauge $\gauge{C}{\cdot}{\cdot}$,
which in turn is used to define Thompson's metric $\thomp$
and Hilbert's projective metric $\hil$ on $C$.

Let $\phi\colon C\to C'$ be a map between two proper open convex cones
in $\linspace$. We say that $\phi$ is \emph{isotone} if
$x\le_C y$ implies $\phi x\le_{C'} \phi y$,
and that is is \emph{antitone} if
$x\le_C y$ implies $\phi y\le_{C'} \phi x$.
We say that $\phi$ is \emph{homogeneous of degree $\alpha\in\R$}
if $\phi(\lambda x)=\lambda^\alpha\phi(x)$, for all $x\in C$ and $\lambda>0$.
Maps that are homogeneous of degree $-1$ we call \emph{anti-homogeneous},
and maps that are homogeneous of degree $1$ we just call \emph{homogeneous}.

For the proofs of the next two propositions,
see~\cite{noll_schaffer_orders_gauge}.

\begin{proposition}
\label{prop:gauge_preserving}
A map $\phi\colon  C\to C'$ is gauge-preserving if and only if it has
any two of the following three properties: isotone, homogeneous,
Thompson-distance preserving.
\end{proposition}

\begin{proposition}
\label{prop:gauge_reversing}
A map $\phi\colon  C\to C'$ is gauge-reversing if and only if it has
any two of the following three properties: antitone, anti-homogeneous,
Thompson-distance preserving.
\end{proposition}

We see from Proposition~\ref{prop:gauge_preserving} that every
linear isomorphism from $C$ to $C'$ is gauge-preserving.
The following theorem shows that the converse is also true.

\begin{theorem}
[\cite{rothaus_order_isomorphisms, noll_schaffer_orders_gauge}]
\label{thm:linear}
Let $\phi\colon  C\to C'$ be a gauge-preserving bijection.
Then, $\phi$ is the restriction to $C$ of a linear isomorphism.
\end{theorem}

\subsection{Hilbert's metric}

\begin{figure}
\input{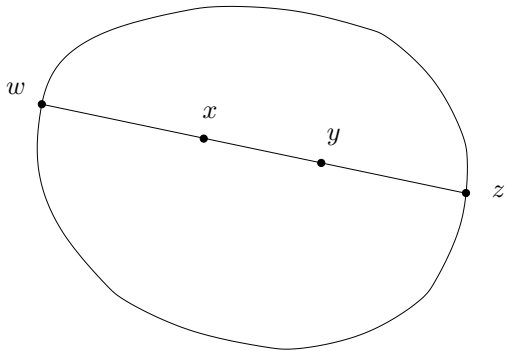}
\caption{Definition of the Hilbert distance.}
\label{fig:hilbert_def}
\end{figure}


Hilbert originally defined his metric on bounded open convex sets.
One can recover his definition by taking a cross section of the cone,
that is, by defining $D:=\{ x\in C \mid f(x)=1 \}$,
where $f\colon \linspace\to\R$ is some linear functional that is positive
with respect to the partial order associated to $C$.
Suppose we are given two distinct points $x$ and $y$ in $D$.
Define $w$ and $z$ to be the points in the boundary $\partial D$ of $D$
such that $w$, $x$, $y$, and $z$ are collinear and arranged in this
order along the line in which they lie.
The Hilbert distance between $x$ and $y$ is then
defined to be the logarithm of the cross ratio of these four points:
\begin{equation*}
\hil(x,y):= \log \frac{|zx|\,|wy|}{|zy|\,|wx|}.
\end{equation*}
On $D$, this definition agrees with the previous one.

If $D$ is an ellipsoid, then the Hilbert metric is Klein's model for
hyperbolic space. At the opposite extreme, if $D$ is an open simplex,
then the Hilbert metric is isometric to a normed space with a
polyhedral unit ball~\cite{delaharpe,nussbaum:hilbert}.

One may of course identify the cross section $D$ with the projective space
$\proj(C)$ of the cone.

Let $(X,d)$ be a metric space and $I\subseteq\R$ an interval.
A map $\gamma\colon I\to X$ is called a \emph{geodesic} if
\begin{align*}
d(\gamma(s),\gamma(t))=|s-t|,
\qquad\text{for all $s,t\in I$}.
\end{align*}
If $I$ is a compact interval $[a,b]$, then the image of $\gamma$ is called a
\emph{geodesic segment} connecting $\gamma(a)$ and $\gamma(b)$.
In the Hilbert geometry, straight-line segments are geodesic segments.
If $I=\R$, then we call the image of $\gamma$ a \emph{geodesic line}.

A subset of $X$ of $\linspace$ is said to be \emph{relatively open}
if it is open in its affine hull.
We denote by $\relint X$ the relative interior of $X$, that is, its
interior, considering it a subset of its affine hull.

A geodesic line is said to be \emph{unique}
if for each compact interval $[s,t]\subset\R$, the geodesic segment
$\gamma([s,t])$ is the only one connecting $\gamma(s)$ and $\gamma(t)$.
The following result characterises the unique geodesic lines~\cite{delaharpe}.

\begin{proposition}
\label{prop:unique_geodesics}
Let $(D,\hil)$ be a Hilbert geometry, and let $w,z\in\partial D$ be such that
the relatively-open line segment $(w,z)$ lies in $D$.
Then, $(w,z)$ is a unique geodesic line if and only if
there is no pair of relatively-open line segments in $\closure D$,
containing $w$ and $z$, respectively, that span a two-dimensional affine space. 
\end{proposition}

Recall that an \emph{exposed face} of a convex set is the intersection
of the set with a supporting hyperplane.
A convex subset $E$ of a convex set $D$ is an \emph{extreme set}
if the endpoints of any line segment in $D$ are contained in $E$ whenever
any point of the relative interior of the line segment is.
The relative interiors of the extreme sets of a convex set $D$ partition $D$.
If an extreme set consists of a single point, we call the point an
\emph{extreme point}. We call a point in the relative interior of a
$1$-dimensional extreme set of the closure of a cone an
\emph{extremal generator} of the cone. Alternatively, an extremal generator
of a cone $C$ is a point $x\in C$ such that $\proj(x)$ is an extreme point
of $\proj(\closure C)$.

\subsection{Symmetric cones}

A proper open convex cone $C$ in a real finite-dimensional vector space
$\linspace$ is called \emph{symmetric} if it is homogeneous and self-dual.
Recall that $C$ is \emph{homogeneous} if its linear automorphism group
$\aut(C) := \{A\in\mathrm{GL}(\linspace)\mid A(C)=C\}$ acts transitively on it,
and it is self-dual if there exists an inner product $\dotprod{\cdot}{\cdot}$
on $\linspace$ for which $C=C^\star$, where
\begin{align*}
C^\star
   := \{ y\in \linspace
               \mid \text{ $\dotprod{y}{x}>0$ for all $x\in\closure{C}$} \}
\end{align*}
is the \emph{open dual} of $C$.
The \emph{characteristic function} $\phi\colon C\to\R$ is defined by
\begin{align*}
\phi(x) =\int_{C^\star} e^{-\dotprod{y}{x}}\,\textrm{d}y,
\qquad\text{for all $x\in C$}.
\end{align*}
This map is homogeneous of degree $-\dim \linspace$,
and so Vinberg's \emph{$*$-map},
\begin{align*}
C\to C^\star,\quad x\mapsto x^*:=-\nabla\log\phi(x),
\end{align*}
is anti-homogeneous.
On symmetric cones, the $*$-map coincides up to a scalar multiple with the
inverse map in the associated Euclidean Jordan algebra~\cite{faraut_koranyi}.
It was shown in~\cite{kai_order_reversing} that on symmetric cones
the $*$-map is antitone. Hence, by Proposition~\ref{prop:gauge_reversing},
it is gauge-reversing for these cones.
It was also shown in~\cite{kai_order_reversing} that the
symmetric cones are the only homogeneous cones for which this is true.

\subsection{The Funk and reverse-Funk metrics}

It will be convenient to consider the Hilbert and Thompson metrics as
symmetrisations of the following function.
Define
\begin{align*}
\funk(x,y) := \log\gauge{C}{x}{y},
\qquad\text{for all $x\in \linspace$ and $y\in C$}.
\end{align*}
We call $\funk$ the \emph{Funk metric} after P.~Funk~\cite{Funk},
and we call its reverse $\rfunk(x,y):=\funk(y,x)$
the \emph{reverse-Funk metric}.

Like Hilbert's metric, the Funk metric was first defined on bounded open
convex sets. On a cross section $D$ of the cone $C$, one can show that
\begin{align*}
\funk(x,y) = \log\frac{|xz|}{|yz|}
\qquad \text{and} \qquad
\rfunk(x,y) = \log\frac{|wy|}{|wx|},
\end{align*}
for all $x,y\in D$. Here $w$ and $z$ are the points of the boundary
$\partial D$ shown in Figure~\ref{fig:hilbert_def}.

On $D$, the Funk metric is a \emph{quasi-metric}, in other words, it satisfies
the usual metric space axioms except that of symmetry.
On $C$, it satisfies the triangle inequality but is not
non-negative. It has the following homogeneity property:
\begin{align*}
\funk(\alpha x, \beta y) = \funk(x,y) + \log\alpha - \log\beta,
\qquad\text{for all $x,y\in C$ and $\alpha,\beta>0$}.
\end{align*}

Observe that both the Hilbert and Thompson metrics are symmetrisations
of the Funk metric:
for all $x,y\in C$,
\begin{align*}
\hil(x,y) &= \funk(x,y) + \rfunk(x,y)
\qquad\text{and} \\
\thomp(x,y) &= \max\big(\funk(x,y), \rfunk(x,y)\big).
\end{align*}

\section{Homogeneity}
\label{sec:homogeneity}

In this section, we will prove that the existence of a gauge-reversing map
on a cone implies that the cone is homogeneous.

\newcommand\id{\operatorname{Id}}

Throughout the paper, we assume that $C$ is a proper open convex
cone in a real finite-dimensional vector space $\linspace$.

\begin{proposition}
\label{prop:gauge_reversing_bijective}
Let $\phi\colon C\to C$ be gauge-reversing. Then, $\phi$ is a bijection.
\end{proposition}

\begin{proof}
By Proposition~\ref{prop:gauge_reversing}, $\phi$ is an isometry of the
Thompson metric. It is therefore injective and continuous.
So, from invariance of domain, we get that $\phi(C)$ is an open set
in~$C$.

Let $y_n$ be a sequence in $\phi(C)$ converging to $y\in C$.
So, there exists a sequence $x_n$ in $C$ such that $\phi(x_n)= y_n$, for all
$n\in\N$. Moreover, $(y_n)$ satisfies the Cauchy criterion, and so $(x_n)$
does also. Therefore, since $(C,\thomp)$ is complete,
$x_n$ converges to some point $x\in C$. From continuity,
we get that $\phi(x)=y$. We have proved that $\phi(C)$ is closed.

Since $C$ is connected and $\phi(C)$ is non-empty and both open and closed,
we conclude that $\phi(C)=C$.
\end{proof}

We use $\id$ to denote the identity operator.

\begin{lemma}
\label{lem:point_fixed}
Let $\phi\colon C\to C$ be a gauge-reversing map that is differentiable at some
point $x\in C$ with derivative $\dee_x\phi = -\id$.
Then, $x$ is a fixed point of $\phi$.
\end{lemma}
\begin{proof}
Since $\phi$ is anti-homogeneous,
we have $\phi(x+\lambda x) = \phi(x) / (1+\lambda)$ for all $\lambda>0$.
This implies that $\dee_x\phi(x) = -\phi(x)$.
But, by hypothesis $\dee_x\phi(x) = -x$. Therefore, $\phi(x)=x$.
\end{proof}

Recall that an involution is a map $\phi$ satisfying $\phi\after\phi=\id$.

\begin{lemma}
\label{lem:almost_all}
Assume there exists a gauge-reversing map $\phi\colon C\to C$.
Then, for almost all $x$ in $C$, there exists a gauge-reversing map
$\phi_x\colon C\to C$ that fixes $x$, has derivative $\dee_x\phi_x = -\id$
at $x$, and is an involution.
\end{lemma}
\begin{proof}
The map $\phi$ is $1$-Lipschitz in the Thompson metric on $C$.
However, this metric is Lipschitz equivalent to the Euclidean metric
on any ball of finite radius in the Thompson metric.
So, we may apply Rademacher's theorem to deduce that $\phi$ is differentiable
almost everywhere within every ball of finite radius, and hence almost
everywhere within all of~$C$.

By Proposition~\ref{prop:gauge_reversing_bijective}, the map $\phi$ is
bijective, and so has an inverse, which is also gauge-reversing.

Let $x$ be a point of $C$ where $\phi$ is differentiable. It follows from
the anti-tonicity of $\phi$ that the linear map $\dee_x\phi$ is antitone,
and hence that $-\dee_x\phi$ is isotone. Similarly, from the anti-tonicity
of $\phi^{-1}$, we deduce that $(-\dee_x\phi)^{-1}=-\dee_{\phi(x)}\phi^{-1}$
is isotone. Therefore, $(-\dee_x\phi)^{-1}$ is a linear isomorphism of $C$,
and hence gauge-preserving.

So the map $\phi_x\colon C\to C$ defined by
$\phi_x:=(-\dee_x\phi)^{-1}\after\phi$ is gauge-reversing.
By the chain rule, $\dee_x\phi_x = -\id$.
So, from Lemma~\ref{lem:point_fixed}, $x$ is a fixed point of $\phi_x$.
The map $\phi_x\after\phi_x$ is gauge-preserving, and therefore linear,
and its derivative at $x$ is $\id$.
We conclude that $\phi_x\after\phi_x=\id$.
\end{proof}

\begin{figure}
\input{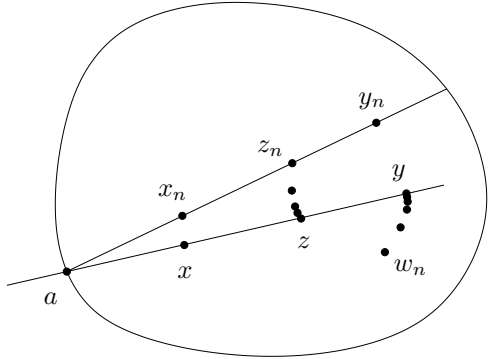}
\caption{Illustration of the proof of Lemma~\ref{lem:exist_collineation}.}
\label{fig:proof_lemma}
\end{figure}

\begin{lemma}
\label{lem:exist_collineation}
Assume there exists a gauge-reversing self-map on $C$.
Let $x$ and $y$ be two points in $\proj(C)$ collinear with an extreme point
of $\proj(\closure C)$.
Then, there exists an element of $\coll(\proj(C))$ mapping $x$ to $y$.
\end{lemma}

\begin{proof}
Let $z$ be the mid-point, in the Hilbert metric on $\proj(C)$, between $x$
and $y$ on the straight line joining them. By Lemma~\ref{lem:almost_all},
we may find a sequence $z_n$ in $\proj(C)$ converging to $z$ such that,
for all $n\in\N$, there is a gauge-reversing map $\phi_{z_n}\colon C\to C$
that fixes some representative $\tilde z_n\in C$ of $z_n$, and has derivative
$-\id$ at $\tilde z_n$. Considering the action of $\phi_{z_n}$ on the
projective space $\proj(C)$, we see that $\phi_{z_n}$
fixes $z_n=\proj(\tilde z_n)$ and its derivative there is $-\id$.

By assumption, there is an extreme point $a$ of $\proj(\closure C)$ such that
$a$, $x$, and $y$ are collinear. We may assume, by relabeling if necessary,
that $x$ lies between $a$ and $y$. For each $n\in\N$, define the
straight line segment $L_n:= az_n\intersection \proj(C)$,
and let $x_n$ and $y_n$ be the two points on $L_n$ satisfying
\begin{align*}
\hil(x_n,z_n) = \hil(z_n,y_n) = \frac{1}{2}\hil(x,y).
\end{align*}
We label these two points in such a way that $x_n$ lies between $a$ and $y_n$;
see Figure~\ref{fig:proof_lemma}. Clearly, $(x_n)$ and $(y_n)$ converge,
respectively, to $x$ and $y$ as $n$ tends to infinity.

Since $az_n$ passes through an extreme point of $\proj(\closure C)$, the line
segment $L_n$ is uniquely geodesic with respect to the Hilbert metric.
So, $\phi_{z_n}(L_n)$ is also uniquely geodesic, and hence a straight line
segment. Using now that $\phi_{z_n}$ fixes $z_n$ and has derivative $-\id$
there, we get that $\phi_{z_n}$ leaves $L_n$ invariant and reverses its
orientation. So we have, $\phi_{z_n}(x_n)=y_n$ for all $n\in\N$.

Again by Lemma~\ref{lem:almost_all}, there exists a sequence $w_n$
in $\proj(C)$ converging to $y$ such that, for all $n\in\N$,
there is a gauge-reversing map $\phi_{w_n}\colon C\to C$ that fixes
a representative $\tilde w_n\in C$ of $w_n$, and therefore fixes $w_n$.
For each $n\in\N$, the map $f_n\colon C\to C$ defined by
$f_n:=\phi_{w_n}\after\phi_{z_n}$ is gauge-preserving, and hence linear,
by Theorem~\ref{thm:linear}. Therefore, the action
of $f_n$ on $\proj(C)$ is in $\coll(\proj(C))$, for each $n\in\N$.

Observe that the sequences $(y_n)$ and $(w_n)$ have the same limit, and that
$\phi_{w_n}(w_n)=w_n$ converges to $y$. So, using that the $\{\phi_{w_n}\}$
are all $1$-Lipschitz, we get that $\phi_{w_n}(y_n)$ converges to $y$.
But $\phi_{w_n}(y_n)=f_n(x_n)$ for all $n\in\N$, and $(x_n)$ converges
to $x$. We conclude that $f_n(x)$ converges to $y$.
This implies that the maps $\{f_n\}$
all lie in some bounded subset of $\coll(\proj(C))$.
It follows that there exists
$f\in\coll(\proj(C))$ such that some subsequence of $(f_n)$ converges
to $f$ uniformly on compact sets of $\proj(C)$. Evidently, $f(x)=y$.
\end{proof}

\begin{lemma}
\label{lem:homogeneous}
Assume there exists a gauge-reversing map $\phi\colon C\to C$.
Then, $C$ is an homogeneous cone.
\end{lemma}
\begin{proof}
Let $x$ and $y$ be points in $C$ such that $y=x+z$, where $z$ is an extremal
generator of $C$. So, $\proj(x)$, $\proj(y)$, and $\proj(z)$ are collinear
in the projective space $\proj(C)$, and $\proj(z)$ is an extreme point
of $\proj(\closure C)$. Therefore, by Lemma~\ref{lem:exist_collineation},
some element of the linear automorphism group $\aut(C)$ maps $x$ to a positive
multiple of $y$. Combining this
automorphism with multiplication by a positive scalar, we can in fact
find an element of $\aut(C)$ that maps $x$ to $y$. The result now follows
since one can get from any element of $C$ to any other by adding
and subtracting a finite number of extremal generators of $C$.
\end{proof}

\section{The Horofunction boundary}
\label{sec:horoboundary}

\newcommand\iso{f}
\newcommand\mapclass{h}
\newcommand\distfn{\psi}
\newcommand\horocyclic{\phi}
\newcommand\dist{d}
\newcommand\symdist{d_{\text{sym}}}
\newcommand\geo{\gamma}
\newcommand\horofunction{h}
\newcommand\isometric{\cong}

We recall in this section the definition of the horofunction boundary,
which will be used extensively in the rest of the paper.
The setting will be that of quasi-metric spaces since some of the metrics
with which we will be dealing, namely, the Funk and reverse-Funk metrics,
are not symmetric.

Let $(X,\dist)$ be a quasi-metric space, that is, a space that satisfies
the usual metric space axioms apart from that of symmetry.
We endow $X$ with the topology induced by the symmetrised metric
$\symdist(x,y):=\dist(x,y)+\dist(y,x)$, which for Funk and reverse-Funk metrics
is the Hilbert metric.

To each point $z\in X$, associate the function
$\distfn_z\colon X\to \R$,
\begin{equation*}
\distfn_z(x) := \dist(x,z)-\dist(b,z),
\end{equation*}
where $b$ is some fixed base-point.
Consider the map $\distfn\colon X\to C(X),\, z\mapsto \distfn_z$ from
$X$ into $C(X)$, the space of continuous real-valued functions
on $X$ endowed with the topology of uniform convergence on bounded sets
of $\symdist$.
This map can be shown to be injective and
continuous~\cite{BGS}.
The \emph{horofunction boundary} is defined to be
\begin{align*}
X(\infty):=\closure\{\distfn_z\mid z\in X\}\backslash\{\distfn_z\mid z\in X\},
\end{align*}
and its elements are called horofunctions.

It is easy to check that the horofunction boundaries obtained using different
base-points are homeomorphic to one another, and that indeed corresponding
horofunctions differ only by an additive constant.

\renewcommand{\theenumi}{\Roman{enumi}}
\renewcommand{\labelenumi}{(\theenumi)}

A geodesic in a quasi-metric space $(X,\dist)$ is a map
$\gamma$ from an interval of $\R$ to $X$ such that
$\dist(\gamma(s),\gamma(t)) = t-s$,
for all $s$ and $t$ in the domain, with $s<t$.
The space $(X,\dist)$ is said to be \emph{geodesic} if for any pair
of points $x$ and $y$ in $X$, there is a geodesic $\gamma\colon[s,t]\to X$
with respect to $d$ that starts at $x$ and ends at $y$.

We make the following assumptions:
\begin{enumerate}
\item
\label{ass:proper}
the metric $\symdist$ is proper, that is, its closed balls are compact;
\item
\label{ass:geodesic}
 $(X,\dist)$ is geodesic;
\item
\label{ass:topology}
for any point $x$ and sequence $x_n$ in $X$, we have $\dist(x_n,x)\to 0$
if and only if $\dist(x,x_n) \to 0$.
\end{enumerate}

These assumptions are satisfied by the Funk and reverse-Funk metrics.

Under assumptions~(\ref{ass:proper}), (\ref{ass:geodesic}),
and~(\ref{ass:topology}), it can be shown that $\distfn$ is an embedding
of $X$ into $C(X)$, in other words, that it is a homeomorphism from $X$
to its image. From now on we identify $X$ with its image.

We will need the next proposition in Section~\ref{sec:product}.

\begin{proposition}
\label{prop:inf_minus_infinity}
Let $(X,d)$ be a proper geodesic metric space. Then $\inf\xi=-\infty$,
for any horofunction $\xi$.
\end{proposition}
\begin{proof}
Let $x_n$ be a sequence converging to $\xi$. Since $X$ is proper, we have that
$d(b,x_n)$ converges to infinity, where $b$ is the base-point.
For each $n\in\N$, let $\gamma_n\colon[0,d(b,x_n)]\to X$ be a geodesic segment
between $b$ and $x_n$. Choose $t>0$. For $n$ large enough,
$d(\gamma_n(t),x_n) = d(b,x_n)-t$. Since the sequence $(\gamma_n(t))_n$ lies
in a compact set, we may, by taking a subsequence if necessary, assume that
it has a limit $y$. Using that the functions $d(\cdot,x_n)-d(b,x_n)$ are
$1$-Lipschitz, we get that
\begin{align*}
\xi(y) = \lim_{n\to\infty} d(\gamma_n(t),x_n)-d(b,x_n)
       = -t.
\end{align*}
The result follows since $t$ is arbitrary.
\end{proof}

Isometries between quasi-metric spaces extend continuously
to homeomorphisms between their horofunction compactifications.
Assume that $\iso$ is an isometry from one quasi-metric
space $(X,d)$ to another $(X',d')$, with base-points $b$ and $b'$,
respectively. Then, for every horofunction $\xi$ and point $x\in X$,
\begin{align*}
\label{prop:transform_horo}
\iso\cdot\xi(x) = \xi(\iso^{-1}(x)) - \xi(\iso^{-1}(b')).
\end{align*}

\subsection{Almost-geodesics and Busemann points}

Let $(X,d)$ be a metric space.
We call a path $\gamma\colon T\to X$,
with $T$ an unbounded subset of $\Rplus$ containing $0$,
an \emph{almost-geodesic} if, for each $\epsilon>0$,
there exists $N\in\R$ such that
\begin{equation*}
|\dist(\gamma(0),\gamma(s))+\dist(\gamma(s),\gamma(t))-t|<\epsilon,
\qquad\text{for all $s$ and $t$ with $N\le s\le t$}.
\end{equation*}

Rieffel proved in ~\cite{rieffel_group} that every almost-geodesic converges.
We say that a horofunction is a \emph{Busemann point}
if there exists an almost-geodesic converging to it,
and denote by $X_B(\infty)$ the set of all Busemann points in $X(\infty)$.

The following alternative characterisation of almost-geodesics will be useful.

\begin{lemma}
\label{lem:alt_almost_geo}
Let $(X,d)$ be a proper geodesic metric space.
A map $\gamma\colon T\to X$, with $T$ an unbounded subset of $\Rplus$
containing $0$, is an almost-geodesic if and only if, given any $\epsilon>0$,
\begin{align}
\label{eqn:alt_geo_1}
|d(\gamma(0),\gamma(s))-s| &< \epsilon
\qquad\text{and} \\
\label{eqn:alt_geo_2}
|d(\gamma(s),\gamma(t)) - t + s| &< \epsilon,
\end{align}
for all $s,t\in T$ large enough, with $s\le t$.
\end{lemma}
\begin{proof}
That~(\ref{eqn:alt_geo_1}) and~(\ref{eqn:alt_geo_2}) hold for any
almost-geodesic was proved by Rieffel~\cite{rieffel_group}.
The implication in the opposite direction is equally straightforward.
\end{proof}

\begin{lemma}
\label{lem:whole_of_Rplus}
Let $(X,d)$ be a proper geodesic metric space,
and let $\xi$ be a Busemann point.
Then, there exists an almost-geodesic defined on the whole of $\Rplus$
that converges to $\xi$.
\end{lemma}

When $X$ is proper and geodesic, one may take $T$ to be $\Rplus$, as the
following lemma demonstrates.

\begin{proof}
Since $\xi$ is a Busemann point, there exists an almost-geodesic
$\alpha\colon T\to X$ converging to it, where $T$ is an unbounded subset of $\Rplus$
containing $0$. Choose a strictly increasing sequence $(t_n)$ in $T$, starting
at $t_0:=0$ and converging to infinity.
Since $(X,d)$ is a geodesic space, we may find, for each $n\in\N$, a geodesic
segment $\beta_n\colon [0,d(\alpha(t_{n}),\alpha(t_{n+1}))]\to X$ from $\alpha(t_n)$
to $\alpha(t_{n+1})$. We interpolate between the points $\alpha(t_n);n\in\N$
by reparametrising these geodesic segments and concatenating them.
Define $\gamma\colon \Rplus\to X$ by
\begin{align*}
\gamma(t):=
   \beta_n\Big( \frac{t-t_n}{t_{n+1}-t_n}
                   d(\alpha(t_{n}),\alpha(t_{n+1})) \Big),
\qquad\text{for all $t\in\Rplus$},
\end{align*}
where $n$ depends on $t$ and is such that $t_n\le t< t_{n+1}$.
Observe that $\gamma(t_n)=\alpha(t_n)$, for all $n\in\N$.

Suppose we are given an $\epsilon>0$.
Let $s$ and $t$ both lie in $[t_n,t_{n+1}]$ for some $n\in\N$. So,
\begin{align*}
\frac{d(\gamma(s),\gamma(t))}{d(\alpha(t_{n}),\alpha(t_{n+1}))}
   = \frac{t-s}{t_{n+1}-t_{n}}
   \le 1.
\end{align*}
From Lemma~\ref{lem:alt_almost_geo}, if $n$ is large enough, then
$|d(\alpha(t_{n}),\alpha(t_{n+1}))-t_{n+1}+t_{n}|<\epsilon/2$.
Therefore, in this case, $|d(\gamma(s),\gamma(t))-t+s|<\epsilon/2$.
This shows that~(\ref{eqn:alt_geo_2}) holds when $s$ and $t$ are large enough
and lie in the same interval $[t_n,t_{n+1}]$.

Now let $p$ lie in $[t_n,t_{n+1}]$ with $n\in\N$. The triangle inequality
gives
\begin{align*}
d\big(\gamma(0),\gamma(t_{n+1})\big) - d\big(\gamma(p),\gamma(t_{n+1})\big)
   &\le d\big(\gamma(0),\gamma(p)\big) \\
   &\le d\big(\gamma(0),\gamma(t_{n})\big)
        + d\big(\gamma(t_{n}),\gamma(p)\big).
\end{align*}
From Lemma~\ref{lem:alt_almost_geo}, both $|d(\gamma(0),\gamma(t_{n}))-t_n|$
and $|d(\gamma(0),\gamma(t_{n+1}))-t_{n+1}|$ are less than $\epsilon/2$,
if $n$ is large enough.
Using this and what we proved in the previous paragraph, we get that,
when $p$ is large enough,~(\ref{eqn:alt_geo_1}) holds, with $p$ substituted
for $s$.

The proof that~(\ref{eqn:alt_geo_2}) holds goes along similar lines.
\end{proof}

In the reverse-Funk metric, one may approach the boundary along a path
of finite length. So, for such spaces we must modify the definition of
almost-geodesic.
We drop the requirement that $T$ be unbounded and instead require that
it that $\sup T$ is a limit point but not an element of $T$.

A path $\gamma\colon T\to X$ is now said to be
an almost-geodesic if, for each $\epsilon>0$,
there exists $N<\sup T$ such that
\begin{equation*}
|\dist(\gamma(0),\gamma(s))+\dist(\gamma(s),\gamma(t))-t|<\epsilon,
\qquad\text{for all $s$ and $t$ with $N\le s\le t$}.
\end{equation*}
One may show again that every almost-geodesic converges, however the limit
may now be a point in $X$.
Again, a Busemann point is a horofunction that is the limit of an
almost-geodesic.
One may verify that most of the results concerning Busemann points
carry over to this new definition.

The Busemann functions provide enough information, in certain cases, to
recover the metric.
We will use the following result in Section~\ref{sec:hilbert_isometries}.

\begin{proposition}
\label{prop:sup_formula}
Let $(X,d)$ be a quasi-metric space satisfying
assumptions~(\ref{ass:proper}), (\ref{ass:geodesic}), and~(\ref{ass:topology}).
Also assume that for each pair of points $x$ and $y$ in $X$ there exists
a geodesic starting at $x$, passing through $y$, and converging to a Busemann
point. Then,
\begin{align*}
d(x,y) = \sup_{\xi} \big(\xi(x) - \xi(y)\big),
\qquad\text{for all $x,y\in X$},
\end{align*}
where the supremum is taken over all Busemann points $\xi$.
\end{proposition}
\begin{proof}
Let $x,y\in X$. Horofunctions are $1$-Lipschitz, that is,
$\xi(x) \le d(x,y) + \xi(y)$ for each horofunction $\xi$.
This implies that the left-hand-side of the equation above is
greater than or equal to the right-hand-side.

To prove the opposite inequality, let $\gamma$ be a geodesic
starting at $x$, passing through $y$, and converging to a Busemann point $\xi$.
For $t>d(x,y)$, we have $d(x,y)+ d(y,\gamma(t)) = d(x,\gamma(t))$.
It follows that $d(x,y) + \xi(y) = \xi(x)$
\end{proof}

\subsection{The detour metric}

We define the \emph{detour cost} for any two horofunctions $\xi$ and $\eta$
in $X(\infty)$ to be
\begin{align*}
H(\xi,\eta) := \sup_{W\ni\xi} \inf_{x\in W\intersection X}
                       \big(d(b,x) + \eta(x) \big),
\end{align*}
where the supremum is taken over all neighbourhoods $W$ of $\xi$ in
$X\union X(\infty)$. 
An equivalent definition is
\begin{align*}
H(\xi,\eta) := \inf_{\gamma} \liminf_{t\to\sup T}
                       \big(d(b,\gamma(t)) + \eta(\gamma(t)) \big),
\end{align*}
where the infimum is taken over all paths $\gamma\colon T\to X$ converging to $\xi$.
This concept first appears in~\cite{AGW-m}.
More detail about it can be found in~\cite{walsh_stretch}.

The detour cost satisfies the triangle inequality and is non-negative.
The Busemann points can be characterised as those horofunctions $\xi$
satisfying $H(\xi,\xi)=0$.

By symmetrising the detour cost, we obtain a metric on the set of Busemann
points:
\begin{align*}
\delta(\xi,\eta):= H(\xi,\eta) + H(\eta,\xi),
\qquad\text{for all Busemann points $\xi$ and $\eta$}.
\end{align*}
We call $\delta$ the \emph{detour metric}. It is possibly infinite valued,
so it actually an \emph{extended metric}.
One may partition the set of Busemann points into disjoint subsets
in such a way that $\delta(\xi,\eta)$ is finite if and only if $\xi$
and $\eta$ lie in the same subset. We call these subsets the \emph{parts}
of the horofunction boundary.

The following expression for the detour cost will prove useful in
Sections~\ref{sec:product} and~\ref{sec:horofunction_thompson}.

\begin{proposition}
\label{prop:detour}
Let $\xi$ be a Busemann point, and $\eta$ a horofunction of a metric space
$(X,d)$. Then,
\begin{align}
\label{eqn:detour}
H(\xi,\eta)
   = \sup_{x\in X}\big(\eta(x)-\xi(x)\big)
   = \inf\big\{\lambda\in\R \mid \eta(\cdot) \le \xi(\cdot) + \lambda \big\}.
\end{align}
\end{proposition}

\begin{proof}
First observe that the second equality of~(\ref{eqn:detour}) is easy to prove.

According to~\cite[Lemma~5.1]{walsh_stretch},
$\eta(\cdot)\le\xi(\cdot)+H(\xi,\eta)$.
This implies that $H(\xi,\eta)$ is greater than or equal to the
right-hand-side of~(\ref{eqn:detour}).

Let $\gamma$ be an almost-geodesic converging to $\xi$.
By~\cite[Lemma~5.2]{walsh_stretch},
\begin{align*}
\lim_{t\to\infty}\big(d(b,\gamma(t))+\eta(\gamma(t))\big)
   &= H(\xi,\eta),
\qquad\text{and} \\
\lim_{t\to\infty}\big(d(b,\gamma(t))+\xi(\gamma(t))\big)
   &= 0.
\end{align*}
Therefore,
$\lim_{t\to\infty}\big(\eta(\gamma(t))-\xi(\gamma(t))\big)=H(\xi,\eta)$.
This shows that $H(\xi,\eta)\le\sup(\eta-\xi)$.
\end{proof}

\section{The horofunction boundary of the Hilbert geometry}
\label{sec:hilbert_horoboundary}

Let $(D,\hil)$ be a Hilbert geometry with base-point $b$.
The horofunction boundary of this geometry is best understood by first
considering the horofunction boundaries of the Funk and the reverse-Funk
geometries.

\subsection{The horofunction boundary of the reverse-Funk geometry}

The following proposition says essentially that the horofunction boundary
of the reverse-Funk geometry is just the usual boundary, and gives an explicit
formula for the horofunctions.
For each $x\in \linspace$, define the function
$r_x(\cdot) := \rfunk(\cdot, x) - \rfunk(b, x)$.

\begin{proposition}[\cite{walsh_hilbert}]
\label{prop:reverseFunk}
The set of horofunctions of the reverse-Funk geometry on $D$ is
$\{r_x \mid x\in\partial D\}$. Every horofunction is a Busemann point.
A sequence in $D$ converges to $r_x$, with $x\in\partial D$, in the horofunction
boundary if and only if it converges to $x$ in the usual topology.
\end{proposition}

We also have a description of the detour metric for this geometry.

\begin{proposition}
[\cite{lemmens_walsh_polyhedral}]
Let $x$ and $y$ be in $\partial D$.
If $x$ and $y$ are in the relative interior of the same extreme set $E$ of
$\closure D$, then the detour metric of the
reverse-Funk geometry is $\delta_R(x,y) = \hil^E(x,y)$,
where $\hil^E$ is the Hilbert metric on $\relint E$.
Otherwise, $\delta_R(x,y)$ is infinite.
\end{proposition}

Thus, the parts of the horofunction boundary are the relative interiors
of the extreme sets of $\closure D$.
We will be particularly interested in parts consisting of a single point.
We call these parts \emph{singletons}.
In the reverse-Funk geometry, there is a singleton
for each extreme point of $\closure D$.

\subsection{The horofunction boundary of the Funk geometry}

The horofunction boundary of the Funk geometry is more complicated than
that of the reverse-Funk geometry. Its Busemann points were worked out
explicitly in~\cite{walsh_hilbert}. We will not need all the details here.

\begin{proposition}
[\cite{walsh_hilbert}]
There is a part associated to every proper extreme set of the polar $D^\circ$
of $D$.
Endowed with its detour metric, each part is isometric to a Hilbert geometry
of the same dimension as the associated extreme set.
If one extreme set $E_1$ is contained in another $E_2$, then the part
associated to $E_1$ is contained in the closure of the part associated
to $E_2$.
\end{proposition}

Let $C$ be the cone over $D$.
Recall that the polar of $D$ may be identified with a cross
section of
\begin{align*}
C^*
   := \{ y\in \linspace^* \mid \text{$\dotprod{y}{x}\ge 0$ for all $x\in C$} \},
\end{align*}
the (closed) dual cone of $C$.

\begin{proposition}
\label{prop:funk_singletons}
[\cite{walsh_hilbert}]
Let $y$ be an extremal generator of $C^*$ normalised so that $\dotprod{y}{b}=1$.
Then, $\log\dotprod{y}{\cdot}$ restricted to $D$ is a singleton of the
Funk geometry of $D$, and every singleton arises in this way.
\end{proposition}

\subsection{The horofunction boundary of the Hilbert geometry}

From the expression of the Hilbert metric as the symmetrisation of the
Funk metric, it is clear that every Hilbert horofunction is the sum of a
Funk horofunction and a reverse-Funk horofunction.
We have the following criterion for convergence.

\begin{proposition}
[\cite{walsh_hilbert}]
A sequence in $D$ converges to a point in the Hilbert-geometry horofunction
boundary if and only if it converges to a horofunction in both the Funk and
reverse-Funk geometries.
\end{proposition}

Combined with Proposition~\ref{prop:reverseFunk},
this implies that every sequence
converging to a Hilbert geometry horofunction also converges to a point
in the usual boundary $\partial D$.

For each $x\in\partial D$, let $B(x)$ denote the set of Funk-geometry
Busemann functions that can be approached by a sequence converging to
$x$ in the usual topology.

\begin{proposition}
[\cite{walsh_hilbert}]
The set of Busemann points of the Hilbert geometry is
\begin{align*}
\{r_x+f \mid \text{$x\in\partial D$ and $f\in B(x)$}\}.
\end{align*}
\end{proposition}

The detour metric was calculated in~\cite{lemmens_walsh_polyhedral}.

\begin{proposition}
[\cite{lemmens_walsh_polyhedral}]
Every part of the horofunction boundary of the Hilbert geometry can be
expressed as the cartesian product of a part of the reverse-Funk geometry
and a part of the Funk geometry.
The detour metric distance between two Hilbert-geometry Busemann points
$h_1 := r_{x_1} + f_1$ and $h_2 := r_{x_2} + f_2$,
with $f_1\in B(x_1)$ and $f_2\in B(x_2)$, is
\begin{align*}
\delta(f_1,f_2)=\delta_R(x_1,x_2) + \delta_F(f_1,f_2),
\end{align*}
where $\delta_R$ and $\delta_F$ are the detour metrics on the set of
Busemann points of, respectively, the reverse-Funk and Funk geometries.
\end{proposition}

So, each part is the $\ell_1$-product of two lower-dimensional Hilbert
geometries.

The following proposition tells us which reverse-Funk parts combine with
with which Funk parts to form a Hilbert part.

\begin{proposition}
[\cite{lemmens_walsh_polyhedral}]
Let $R$ be the reverse-Funk part associated to an extreme set $E$ of $D$,
and let $F$ be the Funk part associated to an extreme set $E'$ of $D^\circ$.
Then $\{r+f \mid r\in R, f\in F\}$ is a part of the Hilbert horofunction
boundary if and only if $\dotprod{y}{x}=0$ for all $y\in E'$ and $x\in E$.
\end{proposition}

Let $U=R\times F$ be a Hilbert part expressed as the product of
a reverse-Funk part and a Funk part. If either of the component parts
is a singleton, then $U$ will itself be a Hilbert geometry when endowed
with its detour metric. We call such parts of the Hilbert horofunction
boundary \emph{pure} parts.
When $R$ is a singleton, $U$ will be called a \emph{pure Funk part}, and
when $F$ is a singleton, $U$ will be called a \emph{pure reverse-Funk part}.

It was shown in~\cite{lemmens_walsh_polyhedral} that the $\ell_1$-product
of two Hilbert geometries, each consisting of more than one point,
can not be isometric to a Hilbert geometry.
It follows that the extension to the horofunction boundary of any isometry
maps pure parts to pure parts. 

Of particular interest will be the \emph{maximal} pure parts.
These are the pure parts that are not contained in the closure of any other
pure part. This property is also preserved by isometries.

Let $W$ be a maximal pure-Funk part. From what we have seen above,
there is some extreme point $w$ of $D$ such that
each element of $W$ can be written $r_w+f$, where $f$ is in the part of the
Funk horofunction boundary associated to the exposed face of $D^\circ$ defined
by $w$.

Similarly, if $U$ is a maximal pure-reverse-Funk part, then any element of
$U$ can be written $r_x+f^{(u)}$, where $f^{(u)}$ is a fixed singleton Funk
horofunction corresponding to some extreme point $u$ of $D^\circ$ that defines
an exposed face $F$ of $D$, and $x\in\relint F$.  

Note that
\begin{align*}
\closure W &= \{ r_w + f \mid f\in B(w) \}
\qquad\text{and} \\
\closure U &= \{ r_x + f^{(u)} \mid f^{(u)}\in B(x) \}.
\end{align*}
Here the closures are taken in the set of Busemann points, not the set of
horofunctions.

\subsection{Horofunctions extended to the cone}
\label{sec:horofunctions_extended}

Recall that Funk and reverse-Funk metrics can be extended to all of $C$.
We may do likewise with the Funk and reverse-Funk horofunctions.
In particular, for $x\in\partial D$,
we have $r_x(y) = \rfunk(y,x)-\rfunk(b,x)$
for all $y\in C$.

Observe that the extension to $C$ of a reverse-Funk horofunction
is the logarithm of an anti-homogeneous function, whereas
that of a Funk horofunction is the logarithm of an homogeneous function.
So, the natural extension to $C$ of any Hilbert horofunction is
homogeneous of degree $0$, that is, constant on the projective class of each
point of $C$.

Recall that every isometry of a metric space extends continuously to
a homeomorphism on the compactification. One may consider a gauge-reversing
map on a cone $C$ to be an isometry from $C$ with the reverse-Funk metric
to $C$ with the Funk metric.
This is formalised in the following proposition, which will be crucial in
our study of gauge-reversing maps.
It says that the extension of a gauge-reversing map to the horofunction
boundary takes reverse-Funk horofunctions to Funk horofunctions
and \emph{vice versa}.
Recall that we may consider $D$ to be a cross-section of the cone $C$.

\begin{proposition}
\label{prop:gauge_reversing_swaps_F_RF}
Let $\phi\colon C\to C$ be a gauge-reversing map.
If $r_x\colon C\to\R$, with $x\in\partial D$,
is an (extended) reverse-Funk horofunction, then
\begin{align*}
\phi r_x(\cdot) := r_x \after \phi^{-1} (\cdot) - r_x \after \phi^{-1}(b)
\end{align*}
is an (extended) Funk horofunction.
Likewise, if $f\colon C\to\R$ is an (extended) Funk horofunction, then
\begin{align*}
\phi f(\cdot) := f \after \phi^{-1} (\cdot) - f \after \phi^{-1}(b)
\end{align*}
is an (extended) reverse-Funk horofunction.
\end{proposition}

\begin{proof}
Let $x_n$ be a sequence in $D$ converging to $r_x$ in the reverse-Funk
horofunction boundary.
For each $n\in\N$, let $z_n\in D$ and $\lambda_n>0$ be such that
$z_n = \lambda_n \phi(x_n)$. For all $n\in\N$, we have
\begin{align*}
\funk(\cdot, z_n) - \funk(b, z_n)
   &= \funk(\cdot, \phi x_n) - \funk(b, \phi x_n) \\
   &= \rfunk(\phi^{-1}(\cdot), x_n) - \rfunk(\phi^{-1}b, x_n).
\end{align*}
So, as $n$ tends to infinity, the sequence $z_n$ converges in the Funk
geometry to the function
$r_x\after \phi^{-1}(\cdot) - r_x\after \phi^{-1}(b)$,
which must therefore be a horofunction.

The proof of the second part is similar.
\end{proof}

\section{Self duality}
\label{sec:self_duality}

\newcommand\dompos{\mathcal{C}}

Max Koecher defined a \emph{domain of positivity} for a symmetric
non-degenerate bilinear form $\bilinear$ on a real finite-dimensional vector
space $\linspace$ to be a non-empty open set $\dompos$ such that $\bilinear(x,y)>0$
for all $x$ and $y$ in $\dompos$, and such that if $\bilinear(x,y)>0$
for all $y\in\closure\dompos\backslash\{0\}$, then $x\in\dompos$.
See~\cite{koecher}

A domain of positivity is always a proper open convex cone.
If the bilinear form is positive definite, then the cone is self-dual.
A negative-definite bilinear form can not have a domain of positivity.

Assume we have a proper open convex cone $C$ in a finite-dimensional vector
space that admits a gauge-reversing map. By Lemma~\ref{lem:almost_all},
there exists a gauge-reversing map $\phi\colon C\to C$ that is an involution,
has a fixed point $b$, and is differentiable at $b$ with derivative
$-\id$. We take $b$ to be the base-point.

We wish to define a positive-definite bilinear form that makes the cone $C$
a domain of positivity.

We use the fact that certain Funk geometry
horofunctions are log-of-linear functions. For an illustration of our method,
consider the positive cone $\interior\R^n_+$ with the gauge-reversing map
$\rho\colon \interior\R^n_+\to\interior\R^n_+$ defined by
$(\rho x)_i := 1/x_i$, for all $x\in\interior\R^n_+$ and coordinates $i$.
The singleton parts of the reverse-Funk horofunction boundary correspond
to the extremal rays of the cone, in this case the (positive) coordinate axes
of $\R^n$.
More precisely, associated to the $i$th coordinate axis $e_i$ is the
reverse-Funk horofunction $r_{e_i}(x)=-\log x_i$.
Each of these singletons is mapped by $\rho$ to a singleton of the Funk
horofunction boundary. In the particular case under consideration, we have
$\rho(r_{e_i})(x) = r_{e_i} \after \rho^{-1}(x) = \log x_i$,
for all $x\in\interior\R^n_+$.
In general, the singleton Funk horofunction is the logarithm of a linear
functional. Thus, we have
a correspondence between the extremal rays of the cone and linear
functionals. This will allow us to define a bilinear form on the space.  

For $x$ an extremal generator of $C$, define
$h_x(\cdot) := \exp(r_x \after \phi(\cdot))$ on $C$, where
$r_x(\cdot):= \log(M(x,\cdot)/M(x,b))$ is the reverse-Funk horofunction
associated to the point $x\in\partial C\backslash\{0\}$.
Observe that $h_x$ is the exponential of the image of $r_x$ under the
map $\phi$.

\begin{lemma}
\label{lem:linear_funcs}
For each extremal generator $x$ of $C$, the function $h_x$ is linear on $C$.
Moreover, $h$ defines a bijection between the projective classes of extremal
generators of $C$ and those of its dual $C^*$.
\end{lemma}

\begin{proof}
Since $\phi$ reverses the gauge, it maps parts of the reverse-Funk
horofunction boundary to parts of the Funk horofunction boundary.
In particular, it maps singletons to singletons.

The parts of the reverse-Funk boundary correspond to the relative interiors
of the extreme sets of $\proj(\closure C)$. So, the singletons correspond
to the projective classes of extremal generators of $C$.

We have seen in Proposition~\ref{prop:funk_singletons} that singletons of the
Funk horofunction boundary correspond to projective classes of extremal
generators of $C^*$, and that each of these horofunctions is the logarithm
of a linear function.
\end{proof}

\begin{definition}
\label{def:dot_product}
For $y$ in $C$, and $x$ an extremal generator of $C$, let
\begin{align*}
\mybilin(y,x) := h_x(y) M(x,b)  = M(x,\phi(y)).
\end{align*}
Extend this definition to all $y\in \linspace$ using Lemma~\ref{lem:linear_funcs}.
We obtain a function on $\linspace$ that is linear in $y$ for every fixed extremal
generator $x$ of $C$.
\end{definition}

Observe that $\mybilin(y,x)$ is homogeneous in $x$.

\begin{lemma}
\label{lem:extreme_extreme}
Let $x$ and $x'$ be extremal generators of $C$.
Then, $\mybilin(x,x')=\mybilin(x',x)$.
\end{lemma}

\begin{proof}
Let $x_n$ and $x'_n$ be sequences in $C$ converging, respectively,
to $x$ and $x'$.
Define
\begin{align*}
j_y(z):= \frac{M(z,y)}{M(b,y)},
\qquad\text{for all $y\in C$ and $z\in \linspace$}.
\end{align*}
Since $x_n$ converges to $x$, the functions $r_{x_n}$ converge pointwise to
the reverse-Funk horofunction $r_x$.
But $\phi$ is a gauge-reversing involution that fixes $b$, and so
\begin{align*}
r_{x_n}\after \phi(\cdot)
   = \log\frac{M(x_n,\phi(\cdot))}{M(x_n,b)}
   = \log\frac{M(\cdot,\phi(x_n))}{M(b,\phi(x_n))}
   = \log j_{\phi(x_n)}(\cdot).
\end{align*}
We deduce that the functions $\log j_{\phi(x_n)}$
converge pointwise to the Funk horofunction $\log h_x$.
Therefore, $j_{\phi(x_n)}$ converges pointwise to $h_x$.
By~\cite[Lemma~3.16]{walsh_hilbert}, the functions $\{j_y\};y\in C$ are
equi-Lipschitzian. We deduce that $j_{\phi(x_n)}(x'_n)$ converges to
$h_x(x')$ as $n$ tends to infinity.
Using in addition that $M(b,\phi(x_n))=M(x_n,b)$ for all $n$,
and that $M(\cdot,b)$ is continuous, we get
\begin{align*}
\lim_{n\to\infty} M(x'_n, \phi(x_n))
   &= \lim_{n\to\infty} j_{\phi(x_n)}(x'_n) M(b,\phi(x_n))\\
   &= h_x(x') M(x,b) \\
   &= \mybilin(x',x).
\end{align*}
But, for each $n\in\N$, we have that $M(x'_n, \phi(x_n))$ equals
$M(x_n, \phi(x'_n))$, and similar reasoning to the above shows that
the limit of this latter quantity is $\mybilin(x,x')$.
\end{proof}

Further extend the definition of $\mybilin$ to $\linspace\times \linspace$ by taking
$\mybilin(y,z) := \sum_j z_j \mybilin(y,x_j)$, where $z=\sum_j z_j x_j$
for some basis of extremal generators $\{x_j\}$ of $C$.

\begin{proposition}
\label{prop:independent}
This definition is independent of the basis of extremal generators chosen.
\end{proposition}
\begin{proof}
Suppose $z=\sum_j z_j x_j= \sum_j z'_j x'_j$ for two bases of extremal
generators $\{x_j\}$ and $\{x'_j\}$. Take any $y\in \linspace$ and write
$y=\sum_j y''_j x''_j$, for some basis of extremal generators $\{x''_j\}$
of $C$. Using Lemma~\ref{lem:extreme_extreme}, we get
\begin{align*}
\sum_j z_j \mybilin(y,x_j)
       &= \sum_{j,k} z_j y''_k \mybilin(x''_k,x_j) \\
       &= \sum_{j,k} z_j y''_k \mybilin(x_j,x''_k) \\
       &= \sum_{k} y''_k \mybilin(z,x''_k).
\end{align*}
Similarly, $\sum_j z'_j \mybilin(y,x'_j)$ can be shown to be equal to the same
expression.
\end{proof}

\begin{lemma}
\label{lem:dop}
The function $\mybilin$ is a symmetric non-degenerate bilinear form,
and $C$ is a domain of positivity for $\mybilin$.
\end{lemma}
\begin{proof}
It is clear from its definition that $\mybilin$ is bilinear, and the symmetry
comes from the bilinearity and Lemma~\ref{lem:extreme_extreme}.

Let $z\in \linspace$ be such that $\mybilin(z,x)=0$, for all $x\in \linspace$.
So, in particular, for all extremal generators $x$ of $C$ we have
$h_x(z) \gauge{}{x}{b} = \mybilin(z,x) = 0$, and hence $h_x(z)=0$.
But, by Lemma~\ref{lem:linear_funcs}, $h_x$ is an extremal generator of
$C^*$, and all extremal generators of $C^*$ arise in this way, up to scale.
Also, $C$ is proper, and so the extremal generators of $C^*$ span $\linspace^*$.
It follows that $z=0$. We deduce that $\mybilin$ is non-degenerate.

If $y\in C$, and $x$ is an extremal generator of $C$, then
$\mybilin(y,x) = M(x,\phi(y))>0$.
It follows that $\mybilin(y,z)>0$ for all $y,z\in C$.

Let $y\in \linspace$ be such that $\mybilin(y,z)>0$
for all $z\in\closure C\backslash\{0\}$.
So, for every extremal generator $x$ of $C$, we have
$h_x(y) = \mybilin(y,x) /  M(x,b) > 0$.
We use again that each $h_x$ is an extremal generator of $C^*$,
and that all extremal generators of $C^*$ arise in this way, up to scale.
We conclude that $y$ is in $C$.
\end{proof}

We must now show that $\mybilin$ is positive definite.

Given a symmetric non-degenerate bilinear form $\bilinear$, diagonalise
it to get a normal basis $\{e_j\}$ such that each $\bilinear(e_j,e_j)$
is either $-1$ or $+1$.
Let $S\colon \linspace\to \linspace$ be the map that changes the sign of each coordinate
associated to a basis element satisfying $\bilinear(e_j,e_j)=-1$.
Also, let $\eucl(x,y) := \bilinear(Sx,y)$ be the Euclidean bilinear form with the
$\{e_j\}$ as an orthonormal basis.
See~\cite{mathoverflow} for a discussion of domains of positivity for
bilinear forms that are not positive definite.

The next lemma uses the following result from~\cite{koecher}:
let $\dompos$ be a domain of positivity with respect to $\bilinear$;
then, $x\in\closure \dompos$ if and only if $\bilinear(x,y)\ge0$
for all $y\in \dompos$.

\newcommand\ystar{z}
\newcommand\yminus{\ystar_-}
\newcommand\yplus{\ystar_+}

\begin{lemma}
\label{lem:dop_boundary}
Let $\dompos$ be a domain of positivity with respect to a symmetric indefinite
non-degenerate bilinear form~$\bilinear$. Then, $\bilinear(z,z)=0$ for some
$z\in\closure \dompos\backslash\{0\}$.
\end{lemma}
\begin{proof}
Define
\begin{align*}
f(y):= \frac{\bilinear(y,y)}{\eucl(y,y)},
\qquad\text{for all $y\in \linspace\backslash\{0\}$}.
\end{align*}
One can calculate that $\grad \eucl(y,y)=2 y$,
and that $\grad \bilinear(y,y)=2S(y)$.
So, the gradient of $f$ is
\begin{align}
\label{eqn:gradient}
\grad f(y) = \frac{2\eucl(y,y)S(y) - 2\bilinear(y,y)y}{\eucl(y,y)^2}.
\end{align}

We wish to minimise $f$ over $\closure \dompos\backslash\{0\}$.
Since $f$ is homogeneous of degree zero and $\closure \dompos\backslash\{0\}$
is projectively compact, the minimum is attained at some point $\ystar$
of $\closure \dompos\backslash\{0\}$.
The minimum is non-negative since $\dompos$ is a domain of positivity.
Let $v:= \grad f(\ystar)$ be the gradient of $f$ at $\ystar$,
and let $\dee_\ystar f$ be the derivative of $f$ at $\ystar$.
These quantities are related by the equation
$\dee_\ystar f(\cdot)= \eucl(v, \cdot)$. Near $\ystar$, we have
\begin{align*}
f(\ystar+\delta\ystar)
   = f(\ystar) + \dee_\ystar f(\delta\ystar) + o(\delta\ystar).
\end{align*}
Since the minimum of $f$ over the convex set $\closure \dompos\backslash\{0\}$
is attained at $\ystar$, we have $\dee_\ystar f(x)\ge 0$
for all $x\in \dompos$, since all such tangent vectors $x$ point into the cone.
So, $\bilinear(Sv, x)=\eucl(v,x)\ge 0$ for all $x\in \dompos$.
It follows that $Sv$ is in $\closure \dompos$.
We conclude that
\begin{align}
\label{eqn:B(v,v)_is_non-negative}
\bilinear(v,v) = \bilinear(Sv,Sv) \ge 0.
\end{align}

Write $\ystar=\yplus + \yminus$, where $\yplus$ is in the linear span
of the basis vectors with $\bilinear(e_j,e_j)=+1$, and $\yminus$ is in
the linear span of the basis vectors with $\bilinear(e_j,e_j)=-1$. So,
\begin{align*}
\eucl(\ystar,\ystar) = \eucl(\yplus,\yplus) + \eucl(\yminus,\yminus)
\quad\text{and}\quad
\bilinear(\ystar,\ystar) = \eucl(\yplus,\yplus) - \eucl(\yminus,\yminus).
\end{align*}

Since $\dompos$ is open and $\bilinear$ is not positive definite,
$\dompos$ contains some element $x$ such that $\bilinear(x,x)<\eucl(x,x)$.
So, from the minimising property of $\ystar$,
we get $\bilinear(\ystar,\ystar)<\eucl(\ystar,\ystar)$,
or, equivalently, $\eucl(\yminus,\yminus)>0$.

We also have that $\ystar\in\closure \dompos$,
and so $\bilinear(\ystar,\ystar)\ge 0$.
Hence $\eucl(\yplus,\yplus)\ge \eucl(\yminus,\yminus)>0$.

One can calculate from~(\ref{eqn:gradient})
and~(\ref{eqn:B(v,v)_is_non-negative}) that
\begin{align*}
0 \le \bilinear(v,v)
  = \frac{16}{\eucl(\ystar,\ystar)^4}
       \eucl(\yminus,\yminus)\eucl(\yplus,\yplus)
          \Big(\eucl(\yminus,\yminus)-\eucl(\yplus,\yplus)\Big).
\end{align*}
So, we see that $\eucl(\yminus,\yminus)\ge \eucl(\yplus,\yplus)$.
In fact equality holds, since we proved the reverse inequality earlier.
We have proved that $\bilinear(\ystar,\ystar)=0$.
\end{proof}

\newcommand\phidot{\varphi}

\begin{lemma}
\label{lem:isolated_fixed}
Let $\phidot\colon C\to C$ be a gauge-reversing map with fixed point $b$.
Then, $b$ is an isolated fixed point of $\phidot$ if and only if $\proj(b)$
is an isolated fixed point of the action of $\phidot$ on $\proj(C)$.
\end{lemma}

\begin{proof}
Let $(x_n)$ be a sequence of points in $\proj(C)$ distinct from $\proj(b)$
that converge to $\proj(b)$ and are fixed by the action of $\phidot$.
Using the anti-homogeneity of $\phidot$, we get that there exists a sequence
$(y_n)$ in $C$ of fixed points of $\phidot$ such that $y_n$ is in the
projective class $x_n$, for each $n\in\N$. The set of elements greater than
or equal to $b$
and the set of elements less than or equal to $b$ are exchanged by $\phidot$,
and the only element they have in common is $b$. Therefore, each $y_n$ is
incomparable to $b$. It follows from this and the projective convergence
of $(y_n)$ to $b$ that $(y_n)$ converges to $b$ in $C$. We have proved that
$b$ is not an isolated fixed point if $\proj(b)$ is not.

The converse is easy.
\end{proof}

\begin{lemma}
\label{lem:unique_fixed}
Let $\phidot\colon C\to C$ be a gauge-reversing map, and consider its action
on the projective space $P(C)$ of the cone.
If $P(b)$ is an isolated fixed point of this action, then $P(b)$ is the unique
fixed point.
\end{lemma}

\begin{proof}
The projective action of $\phidot$ is an isometry of the Hilbert metric $\hil$.
Let $P(b')$ be any fixed point of this projective action.
For each $\alpha\in(0,1)$, the set
\begin{align*}
Z_\alpha := \big\{ z\in P(C) \mid
   \text{$\hil(b,z) = \alpha\hil(b,b')$ and
   $\hil(z,b') = (1-\alpha)\hil(b,b')$} \big\}
\end{align*}
is invariant under $\phidot$.
It is also compact, convex, and non-empty. Therefore, by the Brouwer fixed
point theorem, $Z_\alpha$ contains a fixed point, which will be a distance
$\alpha\hil(b,b')$ from $b$ in the Hilbert metric. Since $\alpha$ can be
made as small as we wish, and we have assumed that $P(b)$ is an isolated
fixed point, applying Lemma~\ref{lem:isolated_fixed}, we see that $P(b')=P(b)$.
\end{proof}

\begin{lemma}
\label{lem:unique_projective_fixed_point}
The projective action of $\phi$ has a unique fixed point in $\proj(C)$.
\end{lemma}

\begin{proof}
The derivative of $\phi$ at $b$ satisfies $D_b\phi= -\id$.
It follows that $b$ is an isolated fixed point of $\phi$.
So, by Lemma~\ref{lem:isolated_fixed}, $P(b)$ is an isolated fixed point
of the projective action.
Applying Lemma~\ref{lem:unique_fixed}, we get the result.
\end{proof}

\begin{lemma}
\label{lem:fixed_between}
Let $z\in C$. Then, $M(z,b) M(b,\phi(z)) = M(z,\phi(z))$.
\end{lemma}

\begin{proof}
Let
\begin{align*}
Y:= \{y\in P(C) \mid
   \text{$\hil(z,y) = \hil(y,\phi(z)) = \hil(z,\phi(z))/2$}
   \}.
\end{align*}   
This set is non-empty, closed, bounded, invariant under $\phi$,
and convex in the usual sense.
So, the Brouwer theorem implies that $Y$ contains a fixed point, which by
Lemma~\ref{lem:unique_projective_fixed_point} must be~$P(b)$.
We deduce that $\hil(z,b) + \hil(b,\phi(z)) = \hil(z,\phi(z))$.
This implies that $\funk(z,b) + \funk(b,\phi(z)) = \funk(z,\phi(z))$.
The result now follows on taking exponentials.
\end{proof}

\begin{lemma}
\label{lem:dot_extreme}
If $x$ is an extremal generator of $C$, then $\mybilin(x,x)=M(x,b)^2$.
\end{lemma}
\begin{proof}
Let $x_n$ be a sequence in $C$ converging to $x$. Using the same reasoning
as in the proof of Lemma~\ref{lem:extreme_extreme}, we get
\begin{align*}
\mybilin(x,x) &= \lim_{n\to\infty} M(x_n,\phi(x_n)).
\end{align*}
So, by Lemma~\ref{lem:fixed_between},
\begin{align*}
\mybilin(x,x)
       &= \lim_{n\to\infty} M(x_n,b) M(b,\phi(x_n)) \\
       &= \lim_{n\to\infty} M(x_n,b)^2 \\
       &= M(x,b)^2.
\qedhere
\end{align*}
\end{proof}

\begin{lemma}
\label{lem:positive_definite}
The bilinear form $\mybilin$ is positive definite.
\end{lemma}
\begin{proof}
We have shown in Lemma~\ref{lem:dop} that $C$ is a domain of positivity of
$\mybilin$. Since $C$ is non-empty, $\mybilin$ can not be negative definite.
Let $y\in\closure C\backslash\{0\}$.
So, we can write $y=\sum_j y_j x_j$ as a positive
combination of finitely many extremal generators $\{x_j\}$ of $C$.
Therefore, $\mybilin(y,y) = \sum_{j,k} y_j y_k \mybilin(x_j,x_k)$.
Since $C$ is a domain of positivity, $\mybilin(x_j,x_k)\ge 0$ for all $j$
and $k$.
Also, $\mybilin(x_j,x_j)>0$ for all $j$ by Lemma~\ref{lem:dot_extreme}.
We conclude that $\mybilin(y,y)>0$.
Therefore, by Lemma~\ref{lem:dop_boundary}, $\mybilin$ is positive definite.
\end{proof}

\begin{lemma}
\label{lem:self_dual}
Assume there exists a gauge-reversing map $\phi\colon C\to C$.
Then, $C$ is self dual.
\end{lemma}
\begin{proof}
As we have seen, by Lemma~\ref{lem:almost_all}, we may assume that
$\phi$ is an involution, has $b$ as a fixed point, and is differentiable at $b$
with derivative $-\id$.
It was shown in Lemma~\ref{lem:dop} that the function $\mybilin(\cdot,\cdot)$
is a symmetric non-degenerate bilinear form,
having $C$ as a domain of positivity.
But $\mybilin$ is positive definite by Lemma~\ref{lem:positive_definite},
and so $C$ is self dual.
\end{proof}

\begin{proof}[Proof of Theorem~\ref{thm:symmetric}]
Assume there exists a gauge-reversing map on $C$. It was proved in
Lemmas~\ref{lem:homogeneous} and~\ref{lem:self_dual} that $C$ is then,
respectively, homogeneous and self dual.

On the other hand, if $C$ is symmetric, then Vinberg's $*$-map
is gauge-reversing, as discussed in the introduction.
\end{proof}

\begin{proof}[Proof of Corollary~\ref{cor:two_cones}]
Suppose $\phi\colon C_1\to C_2$ is a gauge-reversing map between the two cones.
Let $C:=C_1\directproduct C_2$ be the product cone, and define the map $\Phi\colon C\to C$ by
\begin{align*}
\Phi(x_1,x_2):=\big(\phi^{-1}(x_2),\phi(x_1)\big),
\qquad\text{for all $x_1\in C_1$ and $x_2\in C_2$}.
\end{align*}
Since $C$ is a product cone,
\begin{align*}
\gaugebigbracket{C}{(x_1,x_2)}{(y_1,y_2)}
   = \max\Big\{\gauge{C_1}{x_1}{y_1},\gauge{C_2}{x_2}{y_2} \Big\},
\end{align*}
for all $(x_1, x_2)$ and $(y_1, y_2)$ in $C$.
Using this and the fact that both $\phi$ and $\phi^{-1}$ are gauge-reversing,
it is easy to show that $\Phi$ is gauge-reversing.
So, by Theorem~\ref{thm:symmetric}, $C$ is a symmetric cone. It follows that
both $C_1$ and $C_2$ are symmetric.
Vinberg's $*$-map $\inver_{C_2}\colon C_2\to C_2$
on $C_2$ is gauge-reversing. So the map $\inver_{C_2}\after\phi$ is a
gauge-preserving map from $C_1$ to $C_2$, and hence,
by Theorem~\ref{thm:linear}, a linear isomorphism.
\end{proof}

\section{Isometries of the Hilbert metric}
\label{sec:hilbert_isometries}

In this section, we use Theorem~\ref{thm:symmetric}
to determine the isometry group of the Hilbert geometry.


We begin with some Lemmas.

\begin{lemma}
\label{lem:reverse_Funk_unique_geodesic}
Let $(D,\hil)$ be a Hilbert geometry.
Assume there exists a unique-geodesic line connecting some point $\xi$ in the
horofunction boundary to a point $\eta$ in a non-singleton pure-reverse-Funk
part. Then, there exists another unique geodesic line, connecting a point in
the same part as $\xi$ to a point, distinct from $\eta$, in the same part as
$\eta$.
\end{lemma}

\begin{proof}
Let $\gamma$ be the unique geodesic connecting $\xi$ and $\eta$.
The image of $\gamma$ is a relatively-open line segment $xy$ with
$x,y\in\partial D$.
Let $U$ be the part of the horofunction boundary containing~$\eta$.
The point $y$ is contained in the the relative interior $E$ of some extreme set
of $\closure D$.
Since $U$ is a pure-reverse-Funk part, it may be written
$U=\{r_z + f \mid z\in E\}$, where $f$ is some Funk horofunction.
By assumption, $U$ contains a Hilbert horofunction $\eta'$ distinct from $\eta$.
So, we have $\eta' = r_{y'} + f$, for some $y'\in E$ distinct from $y$.
Let $\gamma'$ be the relatively-open line segment $xy'$, parametrised by
arc length in the Hilbert metric.

Since $\gamma$ is uniquely geodesic, there is,
by Lemma~\ref{prop:unique_geodesics}, no pair of relatively-open line segments
in $\closure D$, containing $x$ and $y$, respectively, that span
a two-dimensional affine space. 
Observe that, given any relatively-open line segment in $\closure D$
containing $y'$, we may find a parallel one in $\closure D$ containing $y$. 
Therefore, there is no pair of relatively-open line segments
in $\closure D$, containing $x$ and $y'$, respectively, that span
a two-dimensional affine space. 
We conclude, using Lemma~\ref{prop:unique_geodesics} again,
that $\gamma'$ is uniquely geodesic.

It is not hard to show that $\lim_{t\to-\infty}\hil(\gamma(t),\gamma'(t))$
is finite.
This implies that the points $\xi'$ and $\eta'$ in the horofunction boundary
connected by $\gamma'$ are such that $\xi'$ lies in the same part as $\xi$.
We have already seen that $\eta'$ lies in the same part as $\eta$.
\end{proof}

\begin{lemma}
\label{lem:segment_in_boundary}
Let $\Phi\colon D \to D'$ be a surjective isometry from one Hilbert geometry
$(D,\hil)$ to another $(D',\hil')$,
and let $U$ be a non-singleton maximal pure-Funk part
associated to an extreme point $u$ of $\closure D$.
If $U$ is mapped by $\Phi$ to a reverse-Funk part of $D'$,
then the line segment connecting $u$ to any other extreme point
of $\closure D$ lies in the boundary $\partial D$.
\end{lemma}

\begin{proof}
Suppose there is an extreme point $v$ of $\closure D$ distinct from $u$
such that the
relatively-open line segment $vu$ is contained in $D$.
Let $V$ be the maximal pure-Funk part associated to $v$.
We parameterise $vu$ to get a unit-speed unique-geodesic~$\gamma$.
This geodesic connects some horofunction $\xi$ in $V$ to some horofunction
$\eta$ in $U$.
Observe that any sequence in $D$ converging to a horofunction in $V$
must converge in the usual topology to $v$,
and any sequence converging to a horofunction in $U$ 
must converge in the usual topology to $u$.
Since every unique-geodesic is a parametrised straight line segment,
we conclude that $\gamma$ is the only unique-geodesic, up to
reparameterisation, connecting a horofunction in $V$ to a horofunction in $U$.
So, $\Phi\after\gamma$ is the only unique-geodesic, up to reparameterisation,
connecting a horofunction in $\Phi V$ to a horofunction in $\Phi U$.
Applying Lemma~\ref{lem:reverse_Funk_unique_geodesic}, we get that $\Phi U$
is not a pure-reverse-Funk part. Since it is necessarily pure, it can not
be a reverse-Funk part.
\end{proof}

\begin{lemma}
\label{lem:opposite_types}
Let $W$ and $Z$ be pure parts of the horofunction boundary of a Hilbert
geometry. Assume that $W$ is maximal, that $\closure W$ and $\closure Z$ have a
point in common, and that $Z\not\subset \closure W$. Then, $W$ and $Z$
are of opposite types.
\end{lemma}

\begin{proof}
We consider just the case where $W$ is a pure-Funk part; the other case
is handled similarly. So, $\closure W = \{ r_x + f \mid f \in B(x) \}$,
for some extreme point $x$ of the Hilbert geometry.
We deduce that $\closure Z$ contains a function of the form $r_x + f$,
with $f\in B(x)$. Therefore, if $Z$ was a pure-Funk part,
each of its elements would be of the form $r_x + f$ with $f\in B(x)$,
and $Z$ would be contained in $\closure W$, contrary
to our assumption. We conclude that $Z$ is a pure-reverse-Funk part.
\end{proof}

The following is the key lemma of this section.

\begin{lemma}
\label{lem:def_gaugereversing}
Let $\Phi\colon D \to D'$ be a surjective isometry between two Hilbert
geometries that maps a non-singleton maximal pure-Funk part
to a pure-reverse-Funk part.
Then, $\Phi$ arises as the projective action of a gauge-reversing map
from the cone over $D$ to the cone over $D'$.
\end{lemma}

\begin{proof}
We may assume without loss of generality that $\Phi(b)=b'$, where $b$ and $b'$
are the base-points of $D$ and $D'$, respectively.

Let $U$ be the maximal pure-Funk part in the statement of the lemma,
and let $\Phi(U)$ be its image, which by assumption is a pure-reverse-Funk
part. Associated to $U$ is an extreme point $u$ of $\closure D$, and associated
to $\Phi(U)$ is a Funk horofunction $f^{(u)}$.
So we may write
\begin{align*}
\closure U = \{ r_u + f \mid f\in B(u)\}
\qquad\text{and}\qquad
\closure \Phi(U) = \{ r_x + f^{(u)} \mid f^{(u)}\in B(x)\}.
\end{align*}

Let $C$ and $C'$ be the cones over $D$ and $D'$, respectively, and make the
identifications $\proj(C)=D$ and $\proj(C')=D'$. We extend the Funk,
reverse-Funk, and Hilbert horofunctions to these cones as described in
section~\ref{sec:horofunctions_extended}.

We define a map $\barphi\colon C\to C'$ as follows.
For each $x\in C$, let $\barphi(x)$ be such that
$\proj(\barphi(x))= \Phi(\proj(x))$ and
$f^{(u)}(\barphi(x)) = r_u(x)$.
Clearly, $\Phi$ is the projective action of $\barphi$, and,
from the homogeneity properties of $f^{(u)}$ and $r_u$,
we get that $\barphi$ is anti-homogeneous. Note also that $\phi(b)=b'$.
Moreover, the push-forward of $r_u$ is $f^{(u)}$, that is,
\begin{align}
\label{eqn:pushforward_of_r_u}
\barphi r_u := r_u \after \barphi^{-1} = f^{(u)}.
\end{align}

Let $v$ be any extreme point of $\closure D$ distinct from $u$,
and denote by $V$ the associated maximal pure-Funk part.
So, $\closure V = \{ r_v + f \mid f\in B(v) \}$.
By Lemma~\ref{lem:segment_in_boundary}, the straight line segment
connecting $u$ and $v$ lies in the boundary of $D$.
So, there exists an element of $\partial(C^*)$ that supports $C$ at both
$u$ and $v$. In fact, the set of elements of $\partial(C^*)$ that support $C$
at both $u$ and $v$ forms a proper extreme set of $C^*$, and so contains
an extremal generator $w$ of $C^*$.
Let $E := \relint\{x\in \partial D \mid \dotprod{w}{x}=0 \}$.
So, the maximal pure-reverse-Funk part associated to $w$ can be written
$W:=\{ r_x + f^{(w)} \mid x\in E \}$, where $f^{(w)}$ is the Funk Busemann
point associated to $w$.

Observe that $\dotprod{w}{u}=0$.
So, $\{ h_u := r_u + f^{(w)} \}$ is a singleton part of the Hilbert geometry
horofunction boundary of $D$, and is
contained in both $\closure U$ and $\closure W$.
Therefore, $W$ and $U$ satisfy the assumptions of
Lemma~\ref{lem:opposite_types}, and it follows that $\Phi W$ and $\Phi U$
do also. Applying the Lemma, we get that $\Phi W$ is of opposite type
to $\Phi U$, in other words, $\Phi W$ is a pure-Funk part.

Similarly, $\{ h_v := r_v + f^{(w)} \}$ is a singleton part and is
contained in both $\closure V$ and $\closure W$.
Using the same reasoning as in the previous paragraph, we get that
$\Phi W$ is also of opposite type to $\Phi V$, in other words,
$\Phi V$ is a pure-reverse-Funk part.

Let $z$ be the extreme point of the polar $(D')^\circ$ of $D'$
associated to the maximal pure-Funk part $\Phi W$.
Since $\Phi h_u$ and $\Phi h_v$ lie in $\closure \Phi W$, we may write them
\begin{align}
\label{eqn:phi_h_u}
\Phi h_u &= r_z + f^{(u)}
\qquad\text{and} \\
\label{eqn:phi_h_v}
\Phi h_v &= r_z + f^{(v)},
\end{align}
where $r_z$ is the reverse-Funk horofunction on $D'$ associated to $z$,
and $f^{(u)}$ and $f^{(v)}$ are Funk horofunctions on $D'$.

Observe that if $h$ is any function on $C$ that just depends on
the projective class, then $\barphi h$ is a function on $C'$ that also just
depends on the projective class, and $\barphi h$ agrees with $\Phi h$.
Thus,~(\ref{eqn:phi_h_u}) and~(\ref{eqn:phi_h_v}) hold with $\Phi$
replaced by $\barphi$.
Using this and~(\ref{eqn:pushforward_of_r_u}), we get
\begin{align*}
\barphi f_w
   = \barphi(h_u - r_u)
   = \barphi h_u - \barphi r_u
   = r_z + f^{(u)} - f^{(u)}
   = r_z.
\end{align*}
Therefore, $\barphi r_v = \barphi(h_v - f_w) = f^{(v)}$.

Let $f$ be any Funk horofunction in $B(v)$. So, $h:=r_v + f$ is a Hilbert
horofunction and is contained in the set $\closure V$.
But we have seen that $\Phi V$ is a pure-reverse-Funk part.
It follows that $\barphi h$ is a horofunction of the form $r_p + f^{(v)}$,
for some point $p$ in $\partial D'$.
Therefore, $\barphi f = \barphi(h - r_v) = r_p$.

But $v$ was chosen to be an arbitrary extreme point of $\closure D$,
and every Funk horofunction is contained in $B(v)$ for some choice of
extreme point $v$ of $\closure D$.
So, we have shown that every Funk horofunction is pushed
forward by $\barphi$ to a reverse-Funk horofunction.

\newcommand\funkimage{\funk'}
\newcommand\rfunkimage{\rfunk'}
\newcommand\hilimage{\hil'}

By Proposition~\ref{prop:sup_formula}, we have the following two formulae:
\begin{align}
\label{eqn:funk_sup_formula_for_length}
\funk(x,y) &= \sup_{f} \big(f(x) - f(y)\big),
\qquad\text{for all $x,y\in D$}, \\
\label{eqn:rfunk_sup_formula_for_length}
\rfunkimage(x,y) &= \sup_{r} \big(r(x) - r(y)\big),
\qquad\text{for all $x,y\in D'$},
\end{align}
where the suprema are taken over the set of all Busemann points in,
respectively, the Funk geometry on $D$ and the reverse-Funk geometry on $D'$.
In fact, since the quantities involved have the right homogeneity properties,
the formulae extend to all $x$ and $y$ in $C$ and $C'$, respectively.

So, for every $x,y\in D$ and
every Funk Busemann point $f$ of $D$,
\begin{align*}
f(x) - f(y)
   &= r(\barphi x) - r(\barphi y) \\
   &\le \rfunkimage(\barphi x, \barphi y),
\end{align*}
where $r:=\phi f:=f\after \phi^{-1}$ is the reverse-Funk horofunction of $D'$
that is the push-forward of $f$. We deduce that
$\funk(x,y) \le \rfunkimage(\barphi x, \barphi y)$, for all $x,y\in C$.
Using this inequality and that fact that $\barphi$ preserves the Hilbert
distance, we get the opposite inequality: for all $x,y\in C$,
\begin{align*}
\funk(x,y)
  &= \hil(x,y) - \funk(y,x) \\
  &\ge \hilimage(\barphi x, \barphi y) - \rfunkimage(\barphi y, \barphi x) \\
  &= \rfunkimage(\barphi x, \barphi y).
\end{align*}
Therefore, $\funk(x,y) = \rfunkimage(\barphi x, \barphi y)$,
for all $x,y\in C$.
It follows upon taking exponentials that $\barphi$ is gauge-reversing.
\end{proof}

Let $\Phi\colon D\to D'$ be a surjective isometry between finite-dimensional Hilbert
geometries $D$ and $D'$.
Consider the following property that $\Phi$ may or may not have:
\begin{property}
\label{property:cont_at_extremes}
For every extreme point $u$ of $D$, there is an extreme point $u'$ of $D'$
such that, for all $v\in D$, we have
 $\Phi((u,v)) = (u',\Phi(v))$ as an oriented line-segment,
\end{property}

In~\cite{lemmens_walsh_polyhedral}, it was shown that if $D$ and $D'$ are
polyhedral and $\Phi$ and $\Phi^{-1}$ have
property~\ref{property:cont_at_extremes}, then $\Phi$ is a collineation.
The proof is in two parts. First, it is was shown that $\Phi$ and $\Phi^{-1}$
extend continuously to the usual boundaries of $D$ and $D'$, respectively.
Then it was shown that an isometry between Hilbert geometries that extends
continuously to the boundary and has an inverse that does likewise is
a collineation. Inspecting the proof, one sees that the polyhedral assumption
was not used in any essential way and that the same proof works in the general
case provided one changes some terminology. In particular, one must consider
the extreme points rather than the vertices, and the relative interiors of
the extreme sets rather than the relatively open faces. Thus, we have the
following theorem.
\begin{theorem}[\cite{lemmens_walsh_polyhedral}]
\label{thm:collineation}
Let $\Phi$ be a surjective isometry between two finite-dimensional Hilbert
geometries such that both $\Phi$ and $\Phi^{-1}$ have
Property~\ref{property:cont_at_extremes}. Then, $\Phi$ is a collineation.
\end{theorem}

We now prove the main theorem of this section.

\begin{proof}[Proof of Theorem~\ref{thm:hilbert_isometries}]
Let $\Phi$ be an isometry of $(D,\hil)$, and consider the action of $\Phi$
on the horofunction boundary.
Either every maximal pure-Funk part is mapped to a similar such part,
or there is a non-singleton maximal pure-Funk part that is mapped to a
pure-reverse-Funk part.

In the latter case, $\Phi$ arises as the projective action
of a gauge-reversing self-map on the cone $C$ over $D$,
by Lemma~\ref{lem:def_gaugereversing}.

Consider now the former case. Let $u$ be an extreme point of $D$,
and denote by $U$ the associated maximal pure-Funk part.
By assumption, $U$ is mapped by $\Phi$ to a maximal pure-Funk part
$U':=\Phi(U)$. This part is associated to some extreme point $u'$ of $D$.

Let $v\in D$. By Proposition~\ref{prop:unique_geodesics},
the straight half-line $(u,v)$ is a unique-geodesic.
So, its image $\Phi((u,v))$ is also a unique-geodesic,
and therefore a straight half-line.

Let $x_n$ be sequence in $(u,v)$ converging to $u$ in the usual
topology. Since $x_n$ is moving along a Hilbert geometry geodesic,
it must converge to a Hilbert geometry horofunction, which will be in $U$.
It follows that $\Phi(x_n)$ converges to a Hilbert horofunction in $U'$,
and this horofunction is necessarily of the form $r_{u'} + f$,
with $f\in B(u')$.
This implies that $\Phi(x_n)$ converges to $u'$ in the usual topology.
This establishes that $\Phi$ satisfies Property~\ref{property:cont_at_extremes}.
That $\Phi^{-1}$ satisfies the same property can be shown in the same way.
Applying Theorem~\ref{thm:collineation} gives that $\Phi$ is a collineation,
and so arises as the projective action of a gauge-preserving self-map of $C$.
\end{proof}

\begin{proof}[Proof of Corollary~\ref{cor:hilbert_isometries}]
Denote by $\Lambda^+$ the set of self-maps of $C$ that are gauge-preserving,
and by $\Lambda$ the set that are either gauge-preserving or gauge-reversing.
By Theorem~\ref{thm:hilbert_isometries}, every isometry of $D$ arises as the
projective action of a map in $\Lambda$.

If the cone $C$ is not symmetric, then by Theorem~\ref{thm:symmetric},
$\Lambda$ consists of only gauge-preserving maps, and so every isometry
is a collineation.

So, assume that $C$ is symmetric.

Observe that the composition of a gauge-reversing map and a gauge-preserving
map is gauge-reversing, and that the composition of two gauge-reversing maps
is gauge-preserving,
It follows easily that $\Lambda^+$ is a normal subgroup of index two in
$\Lambda$.
This implies that $\Lambda$ is generated by $\Lambda^+$ and the Vinberg
$*$-map associated to $C$, which is always gauge-reversing for symmetric cones.
So, $\isom(D)$ is generated by the collineations and the projective
action of the $*$-map.

If $C$ is Lorentzian, then the projective action of the $*$-map
is a collineation, so in this case $\isom(D)=\coll(D)$.

For all other symmetric cones, this projective action is not a collineation,
and therefore $\coll(D)$ is a normal subgroup of index two in $\isom(D)$.
\end{proof}

\section{Horofunction boundary of product spaces}
\label{sec:product}

Let $(X_1,d_1)$ and $(X_2,d_2)$ be metric spaces.
We define the $\ell_\infty$-product of these two spaces to be the space
$X:=X_1\times X_2$ endowed with the metric $d$ defined by
\begin{align*}
d\big((x_1,x_2),(y_1,y_2)\big) := \max\big\{d_1(x_1,y_1), d_2(x_2,y_2)\big\},
\end{align*}
for all $x_1,y_1\in X_1$, and $x_2,y_2\in X_2$, and we denote this space
\begin{align*}
(X,d) =: (X_1,d_1) \oplus_\infty (X_2,d_2).
\end{align*}

Our motivation for considering such spaces is that the Thompson metric
on a product cone has such a structure, a fact we will use when
studying the isometry group of the Thompson metric.

In this section, we will study the set of Busemann points and the detour cost
for $\ell_\infty$-product spaces.
We assume that $X_1$ and $X_2$ have base-points $b_1$ and $b_2$, respectively,
and we take $(b_1,b_2)$ to be the base-point of $X$.

\newcommand\Raug{\overline\R}

Let $\vee$ and $\wedge$ denote, respectively, maximum and minimum.
We use the convention that addition and subtraction take
precedence over these operators.
We write $x^+:=x\vee 0$ and $x^-:=x\wedge 0$.
Let $\Raug:=\R\union\{-\infty,+\infty\}$.
Given two real-valued functions $f_1$ and $f_2$,
and $c\in\Raug$, define
\begin{align*}
[f_1,f_2,c] := f_1 + c^- \vee f_2 - c^+.
\end{align*}

For the rest of this section, we will assume that $(X_1,d_1)$ and $(X_2,d_2)$
are proper geodesic metric spaces.

The following proposition shows that horofunctions of $\ell_\infty$-product
spaces have a simple form.

\begin{proposition}
\label{prop:horofunction_product}
Every horofunction of $(X,d)$ is of the form $[\xi_1,\xi_2,c]$,
with $\xi_1\in X_1(\infty)$ and $\xi_2\in X_2(\infty)$, and $c\in\Raug$.
\end{proposition}

\begin{proof}
Denote by $\overline X_1$ and $\overline X_2$ the horofunction
compactifications of $X_1$ and $X_2$, respectively.
Let $x^n=(x_1^n, x_2^n)$ be a sequence in $X$ converging to a horofunction
$\xi$. By passing to a subsequence if necessary, we may assume
that $x_1^n$ converges to $\xi_1\in\overline X_1$,
that $x_2^n$ converges to $\xi_2\in\overline X_2$, and that
\begin{align*}
d(b_1, x_1^n) - d(b_2, x_2^n) \to c,
\end{align*}
with $c\in\Raug$.
At least one of $\xi_1$ and $\xi_2$ must be a horofunction of its respective
space. Observe that, as $n$ tends to infinity,
\begin{align*}
d(b,x^n) - d(b_1, x_1^n) &\to -c^-,
\qquad\text{and} \\
d(b,x^n) - d(b_2, x_2^n) &\to c^+.
\end{align*}
Therefore, for $y=(y_1, y_2)$ in $X$, we have the following limit as $n$
tends to infinity:
\begin{align*}
d(y, x^n) - d(b, x^n)
   &= \big( d(y_1, x_1^n) \vee d(y_2, x_2^n) \big) - d(b, x^n) \\
   &= \big(d(y_1, x_1^n) - d(b_1, x_1^n) + c^- \big)
         \vee \big(d(y_2, x_2^n) - d(b_2, x_2^n) - c^+ \big) \\
   &\to [\xi_1, \xi_2, c](y).
\end{align*}
If $\xi_1$ is not a horofunction, then $c=-\infty$, and $\xi_1$ is irrelevant
in the expression $[\xi_1, \xi_2, c]$.
Likewise, if $\xi_2$ is not a horofunction, then $c=+\infty$,
and $\xi_2$ is irrelevant.
\end{proof}

Next, we will determine the Busemann points of product spaces.
We will need the following lemma.

\begin{lemma}
\label{lem:unique_split}
Let $f_1$ and $g_1$ be real-valued functions on a set $Y_1$,
and let $f_2$ and $g_2$ be real-valued functions on a set $Y_2$.
Assume that $f_1(x_1) \vee f_2(x_2) = g_1(x_1) \vee g_2(x_2)$ for all
$(x_1,x_2)\in Y_1\times Y_2$, and that $\inf f_2 = \inf g_2 = -\infty$.
Then, $f_1=g_1$.
\end{lemma}

\begin{proof}
Let $x_1\in Y_1$. Choose $x_2\in Y_2$ such that $f_2(x_2)<f_1(x_1)$. So,
\begin{align*}
f_1(x_1)=f_1(x_1) \vee f_2(x_2) =  g_1(x_1) \vee g_2(x_2) \ge g_1(x_1).
\end{align*}
The reverse inequality is proved similarly.
\end{proof}

We will also need the following characterisation of Busemann points
from~\cite[Theorem~6.2]{AGW-m}: a horofunction is a Busemann point
if and only if it can not be written as
the minimum of two $1$-Lipschitz functions, both different from it.

\begin{proposition}
\label{prop:busemann_product}
For every pair of Busemann points $\xi_1\in X_1(\infty)$ and
$\xi_2\in X_2(\infty)$, and every $c\in\Raug$, the function
$[\xi_1,\xi_2,c]$ is a Busemann point of $X$. Moreover, every Busemann point
of $X$ arises in this way.
\end{proposition}

\begin{proof}
Let $\xi$ be a Busemann point of $X$. By Lemma~\ref{prop:horofunction_product},
we may write $\xi =[\xi_1,\xi_2,c]$, with $\xi_1\in X_1(\infty)$ and
$\xi_2\in X_2(\infty)$, and $c\in\Raug$.
Suppose $\xi_1=f \wedge f'$, where $f$ and $f'$ are real-valued $1$-Lipschitz
functions on $(X_1,d_1)$. We consider the case when $c\ge0$; the other case
is similar. So,
\begin{align*}
\xi &= (f \wedge f') \vee (\xi_2 - c) \\
    &= \big(f \vee (\xi_2-c)\big) \wedge \big(f' \vee (\xi_2-c)\big).
\end{align*}
Therefore, $\xi$ is the minimum of two $1$-Lipschitz functions on $(X,d)$.
This implies, since $\xi$ is Busemann, that it is equal to one of them,
say $\xi=f \vee (\xi_2-c)$.
Applying Lemma~\ref{lem:unique_split}, we get that $\xi_1$ equals $f$.
We conclude that $\xi_1$ is a Busemann point of $(X_1,d_1)$.
The proof that $\xi_2$ is Busemann is similar.

Now assume that $\xi_1$ and $\xi_2$ are Busemann points of $(X_1,d_1)$ and
$(X_2,d_2)$, respectively. So, there exists an
almost-geodesic $\gamma_1$ in $X_1$ converging to $\xi_1$, and an
almost-geodesic $\gamma_2$ in $X_2$ converging to $\xi_2$.
We may assume without loss of generality that $\gamma_1$ and $\gamma_2$
start, respectively, at $b_1$ and $b_2$, the base-points of the spaces.
By Lemma~\ref{lem:whole_of_Rplus}, we may also assume that the domain of
definition of these almost-geodesics is $\Rplus$. We furthermore assume that
$c\ge0$; the other case is handled similarly.

Define the path $\gamma\colon \Rplus\to X$ by
\begin{align*}
\gamma(t) := \Big( \gamma_1(t), \gamma_2\big((t-c)^+\big)\Big),
\qquad\text{for all $t\in\Rplus$}.
\end{align*}
By Lemma~\ref{lem:alt_almost_geo}, we have, as $t$ tends to infinity,
\begin{align}
\label{eqn:busemann_param_a}
d_1\big(b_1,\gamma_1(t)\big)-t &\to 0, \qquad\text{and} \\
\label{eqn:busemann_param_b}
d_2\big(b_2,\gamma_2(t)\big)-t &\to 0.
\end{align}
Observe that $(t-c)^+ -t$ converges to $-c$ as $t$ tends to infinity.
We deduce from this and~(\ref{eqn:busemann_param_b}) that
\begin{align*}
d_2\big(b_2,\gamma_2((t-c)^+)\big)-t \to -c,
\qquad\text{as $t\to\infty$}.
\end{align*}
From this and~(\ref{eqn:busemann_param_a}), we get
\begin{align}
\label{eqn:busemann_param}
d(b,\gamma(t))-t \to 0,
\qquad\text{as $t\to\infty$}.
\end{align}
This shows that condition~(\ref{eqn:alt_geo_1}) of
Lemma~\ref{lem:alt_almost_geo} holds for $\gamma$.
The proof of condition~(\ref{eqn:alt_geo_2}) of the same lemma is similar.
So, $\gamma$ is an almost-geodesic.

Using~(\ref{eqn:busemann_param}), we get, for all $x:=(x_1,x_2)$ in $X$,
\begin{align*}
\lim_{t\to\infty}\Big(d(x,\gamma(t))-d(b,\gamma(t))\Big)
   &= \lim_{t\to\infty}
        \Big(\big(d_1(x_1,\gamma_1(t)) \vee d_2(x_2,\gamma_2((t-c)^+))\big)- t\Big) \\
   &= \lim_{t\to\infty}
        \Big(\big(d_1(x_1,\gamma_1(t)) - t \big)
           \vee \big(d_2(x_2,\gamma_2((t-c)^+))- t\big)\Big) \\
   &= \xi_1(x_1) \vee \big(\xi_2(x_2) -c\big).
\end{align*}
In other words, $\gamma(t)$ converges to $\xi:=[\xi_1,\xi_2,c]$.
Therefore, $\xi$ is a Busemann point of $(X,d)$.
\end{proof}

We now calculate the detour cost in product spaces.
We use the convention that $-\infty$ is absorbing for addition,
that is, $(+\infty)+(-\infty)=-\infty$ and $(-\infty)-(-\infty)=-\infty$.

\begin{proposition}
\label{prop:detour_product}
Let $\xi=[\xi_1,\xi_2,u]$ and $\eta=[\eta_1,\eta_2,v]$ be Busemann points of
$X$. Then,
\begin{align*}
H(\xi,\eta) = \max\Big(H(\xi_1,\eta_1)-u^-+v^-, H(\xi_2,\eta_2)+u^+-v^+\Big).
\end{align*}
\end{proposition}

\begin{proof}
We extend the definition of $H$ somewhat by letting
$H(\xi+u,\eta+v):=H(\xi,\eta)+v-u$ for all
Busemann points $\xi$ and $\eta$, and $u,v\in[-\infty,0]$.

Write
\begin{align*}
\bar\xi_1 &:= \xi_1 + u^-, & \bar\xi_2 &:= \xi_2 - u^+, \\
\bar\eta_1 &:= \eta_1 + v^-, & \bar\eta_2 &:= \eta_2 - v^+. \\
\end{align*}
So, $\xi=\bar\xi_1 \vee \bar\xi_2$ and $\eta=\bar\eta_1 \vee \bar\eta_2$.

By Proposition~\ref{prop:detour},
\begin{align*}
\eta_1(\cdot) - \xi_1(\cdot) &\le H(\xi_1,\eta_1),
\qquad\text{and} \\
\eta_2(\cdot) - \xi_2(\cdot) &\le H(\xi_2,\eta_2).
\end{align*}
Therefore,
\begin{align*}
\bar\eta_1(\cdot) - \bar\xi_1(\cdot) &\le H(\bar\xi_1,\bar\eta_1),
\qquad\text{and} \\
\bar\eta_2(\cdot) - \bar\xi_2(\cdot) &\le H(\bar\xi_2,\bar\eta_2).
\end{align*}
Let $M:=\max\big( H(\bar\xi_1,\bar\eta_1), H(\bar\xi_2,\bar\eta_2) \big)$.
If $M=+\infty$, then clearly $H(\xi,\eta)\le M$, so assume that $M<+\infty$.
This implies that it is not the case that $u=-\infty$ and $v>-\infty$,
nor that $u=+\infty$ and $v<+\infty$. It follows that
\begin{align*}
\bar\eta_1(\cdot) &\le H(\bar\xi_1,\bar\eta_1) + \bar\xi_1(\cdot),
\qquad\text{and} \\
\bar\eta_2(\cdot) &\le H(\bar\xi_2,\bar\eta_2) + \bar\xi_2(\cdot).
\end{align*}
Therefore, using Proposition~\ref{prop:detour},
\begin{align*}
H(\xi,\eta)
   &= \sup_{(x_1,x_2)\in X}
          \Big( \big(\bar\eta_1(x_1) \vee \bar\eta_2(x_2) \big)
       -  \big(\bar\xi_1(x_1) \vee \bar\xi_2(x_2) \big) \Big) \\
   &\le \sup_{(x_1,x_2)\in X}
          \Big( \big(H(\bar\xi_1,\bar\eta_1)+\bar\xi_1(x_1)
             \vee H(\bar\xi_2,\eta_2)+\bar\xi_2(x_2) \big)
       -  \big(\bar\xi_1(x_1) \vee \bar\xi_2(x_2) \big) \Big) \\
   &\le \sup_{(x_1,x_2)\in X}
          \Big( \big(M+\bar\xi_1(x_1)
             \vee M+\bar\xi_2(x_2) \big)
       -  \big(\bar\xi_1(x_1) \vee \bar\xi_2(x_2) \big) \Big) \\
   &= M.
\end{align*}

We now wish to prove the reverse inequality. We have
\begin{align*}
H(\xi,\eta)
   \ge \sup_{(x_1,x_2)\in X} \Big( \bar\eta_1(x_1)
             - \big( \bar\xi_1(x_1) \vee \bar\xi_2(x_2) \big)
       \Big).
\end{align*}
Fix $x_1\in X_1$. From Proposition~\ref{prop:inf_minus_infinity} and the fact
that $u^+\ge0$, we get that $\inf\bar\xi_2=-\infty$. Therefore, we can choose
$x_2\in X_2$ to make $\bar\xi_2(x_2)$ as negative as we wish. So we see that
$H(\xi,\eta)\ge \bar\eta_1(x_1) - \bar\xi_1(x_1)$.
We conclude that
\begin{align*}
H(\xi,\eta)
   \ge \sup_{x_1\in X_1} \Big( \bar\eta_1(x_1) - \bar\xi_1(x_1) \Big)
   = H(\bar\xi_1,\bar\eta_1).
\end{align*}
Similar reasoning shows that $H(\xi,\eta) \ge H(\bar\xi_2,\bar\eta_2)$.
\end{proof}

Using this proposition, we can characterise the singletons of product
spaces.

\begin{corollary}
\label{cor:product_singletons}
The following are equivalent:
\begin{itemize}
\item
$\xi$ is a singleton Busemann point in the horofunction boundary of $X$;
\item
$\xi$ takes one of the following two forms:
$\xi(x_1,x_2) = \xi_1(x_1)$ with $\xi_1$ a singleton Busemann point of $X_1$,
or $\xi(x_1,x_2) = \xi_2(x_2)$ with $\xi_2$ a singleton Busemann point of
$X_2$.
\end{itemize}
\end{corollary}

\begin{proof}
Let $\xi$ be a Busemann point of $X$.
By Proposition~\ref{prop:busemann_product}, we can write $\xi=[\xi_1,\xi_2,c]$,
where $\xi_1$ and $\xi_2$ are Busemann points of $X_1$ and $X_2$, respectively,
and $c\in\Raug$.

Consider the case where $c$ is finite.
For each $\epsilon\in\R$, define the Busemann point
$\xi^\epsilon := [\xi_1,\xi_2,c+\epsilon]$.
We can calculate from Proposition~\ref{prop:detour_product} that
$H(\xi,\xi^\epsilon) + H(\xi^\epsilon,\xi) = |\epsilon|$, for all
$\epsilon\in\R$. This shows that $\xi^\epsilon$ is distinct from $\xi$,
but in the same part as it, for all $\epsilon\in\R\backslash\{0\}$,
and hence that $\xi$ is not a singleton.

In the case where $c=\infty$, we have $\xi=\xi_1$. If $\xi_1$ is not a
singleton of $X_1$, then let $\xi'_1$ be another Busemann point in the same
part as $\xi_1$, and write $\xi'=\xi'_1$.
From Proposition~\ref{prop:detour_product}, we get
$H(\xi,\xi')+H(\xi',\xi) = H(\xi_1,\xi'_1) + H(\xi'_1,\xi_1)<\infty$,
and so $\xi$ is not a singleton.

The case where $c=\infty$ and $\xi_2$ is not a singleton of $X_2$
is handled similarly.

Now let $\xi_1$ be a singleton Busemann point of $X_1$,
and write $\xi(x_1,x_2) := \xi_1(x_1)$ for all $x_1\in X_1$ and $x_2\in X_2$.
Observe that $\xi=[\xi_1,\xi_2,\infty]$, for any Busemann point $\xi_2$
of $X_2$. Let $\eta:=[\eta_1,\eta_2,c]$ be a Busemann point of $X$ in the same
part of the horofunction boundary as $\xi$.
Here, of course, $\eta_1$ and $\eta_2$ are Busemann points of $X_1$ and $X_2$,
respectively, and $c\in\Raug$. 
Since $H(\xi,\eta)$ is finite,
looking at Proposition~\ref{prop:detour_product}, we see that
$H(\xi_2,\eta_2)+\infty-c^+ < \infty$, and hence $c=\infty$.
Using this and the same proposition again, we get
$H(\xi_1,\eta_1) = H(\xi,\eta)<\infty$ and
$H(\eta_1,\xi_1) = H(\eta,\xi)<\infty$.
Since $\xi_1$ was assumed to be a singleton, we conclude that
$\eta_1=\xi_1$, and therefore $\eta=\xi$.
We have thus shown that $\xi$ is a singleton.

That $\xi:=\xi_2$ is a singleton point of $X$ whenever
$\xi_2$ is singleton point of $X_2$ may be proved in a similar manner.
\end{proof}

\section{Horofunction boundary of the Thompson metric}
\label{sec:horofunction_thompson}

\newcommand\tmetric{d_T}

In this section, we determine the horofunction boundary of the Thompson
geometry and its set of Busemann points. We then calculate the detour metric
on the boundary.

The results of this section will resemble somewhat those of the last.
This is because the Thompson metric is the maximum of the Funk and
reverse-Funk metrics, and, as consequence, its boundary is related to those
of these two metrics
in a way similar to how the boundary of a $\ell_\infty$-product space
is related to the boundaries of its components.

For each $x\in D$, let $A(x)$ denote the set of horofunctions of the Funk
geometry that may be approached by a sequence in $D$ converging to $x$.
Also, for each $x\in D$, define the following function on $C$:
\begin{align*}
f_{C,x}(\cdot):= \log\frac{\gauge{C}{\cdot}{x}}{\gauge{C}{b}{x}}.
\end{align*}

We start off by describing the horofunctions of the Thompson metric..

\begin{proposition}
\label{prop:horo_thompson}
Let $(C,\tmetric)$ be a proper open convex cone with its Thompson metric.
Its horofunction boundary is
\begin{align*}
C(\infty) =
   \{r_x \mid x\in D\}
   \union \{f_{C,x} \mid x\in D\}
   \union\{ [r_x,f,c] \mid x\in \partial D, f\in A(x), c\in \Raug \},
\end{align*}
where $D$ is a cross section of $C$.
\end{proposition}

\begin{proof}
First, we show that each of the functions in the statement is a horofunction.
Functions of the form $r_x$, with $x\in D$, may be approached by taking
$\lambda x$ as $\lambda>0$ tends to infinity.
Similarly, $f_{C,x}$, with $x\in D$, is approached by $x/\lambda$
as $\lambda$ tends to infinity.

Let $x\in\partial D$, and $f\in A(x)$, and $c\in\Raug$.
So, there exists a sequence $x_n\in\proj(C)$ such that $x_n$ converges to $x$
and $\funk(\cdot,x_n)-\funk(b,x_n)$ converges pointwise to $f$.
For each $n$, we may choose in $C$ a representative $y_n$ of $x_n$ such that
$\gauge{}{b}{y_n} = \gauge{}{y_n}{b}$. It is not difficult to show that
the limit of the sequence $y_n\exp(c_n/2)$ in the horofunction
compactification is $[r_x,f,c]$, for any any sequence $c_n$
in $\R$ converging to $c$.

Now we show that all horofunctions take one of the given forms.
Let $y_n$ be a sequence in $C$ converging to a horofunction $\xi$.
Using compactness, we may assume that $y_n$ converges in both the Funk
and reverse-Funk horofunction compactifications.
If $y_n$ converges projectively to the projective class of some point
$y\in D$, then $\xi$ must equal $r_y$ if $y_n$ heads away from the origin,
or $f_{C,y}$ if $y_n$ heads towards the origin.
Otherwise, $y_n$ converges to a horofunction $f$ in the Funk geometry,
and to a horofunction $r_x$ in the reverse-Funk geometry,
with $x\in\partial D$ and $f\in A(x)$.
By taking a subsequence if necessary, we may
assume that $\rfunk(b,y_n) - \funk(b,y_n)$ converges to a limit $c$
in $\Raug$. One may calculate then that $\tmetric(\cdot,y_n)-\tmetric(b,y_n)$
converges to $[r_x, f, c]$ as $n$ tends to infinity.
\end{proof}

The following lemma parallels Lemma~\ref{lem:unique_split}.

\begin{lemma}
\label{lem:unique_split_on_cone}
Let $f_1$, $f_2$, $g_1$, and $g_2$ be real-valued functions on a cone $C$
satisfying
\begin{align*}
f_1(\lambda x) &= -\log\lambda+f_1(x),
   & f_2(\lambda x) &= \log\lambda+f_2(x), \\
g_1(\lambda x) &= -\log\lambda+g_1(x),
   & g_2(\lambda x) &= \log\lambda+g_2(x),
\end{align*}
for all $\lambda>0$ and $x\in C$.
Assume that $f_1 \vee f_2 = g_1 \vee g_2$ on $C$.
Then, $f_1=g_1$ and $f_2=g_2$.
\end{lemma}

\begin{proof}
Let $x\in C$, and choose $\lambda>0$ small enough that
$f_2(\lambda x)<f_1(\lambda x)$. So,
\begin{align*}
f_1(\lambda x)
   = f_1(\lambda x) \vee f_2(\lambda x)
   =  g_1(\lambda x) \vee g_2(\lambda x) \ge g_1(\lambda x).
\end{align*}
Therefore, $f_1(x) \ge g_1(x)$.
The reverse inequality is proved similarly.

The proof that $f_2=g_2$ goes along the same lines.
\end{proof}

In~\cite{AGW-m}, the notion of almost-geodesic was defined slightly
differently. A sequence $(x_k)$ in a metric space $(X,d)$ was said
to be an \emph{$\epsilon$-almost-geodesic} if
\begin{align*}
d(x_0,x_1) + \cdots + d(x_m,x_{m+1}) \le d(x_0,x_{m+1}) + \epsilon,
\qquad\text{for all $m\in\N$.}
\end{align*}
It was shown in~\cite{AGW-m} that every $\epsilon$-almost-geodesic has
a subsequence that may be parameterised to give an almost-geodesic
in the sense of Rieffel. Conversely, given any almost-geodesic
in the sense of Rieffel, it was shown that one may obtain an
$\epsilon$-almost-geodesic by taking a sequence of points along it.

Recall again that a horofunction is a Busemann point if and only if it
can not be written as the minimum of two $1$-Lipschitz functions,
both different from it~\cite[Theorem~6.2]{AGW-m}.
In the context of a distance $d$ that is not symmetric, a function $f$
being $1$-Lipschitz means that $f(x) \le d(x,y) + f(y)$,
for all points $x$ and $y$.

\begin{proposition}
\label{prop:busemann_thompson}
The set of Busemann points of the Thompson geometry is
\begin{align*}
\{r_x \mid x\in D\}
   \union \{f_{C,x} \mid x\in D\}
   \union\{ [r_x,f,c] \mid x\in \partial D, f\in B(x), c\in \Raug \}.
\end{align*}
\end{proposition}

\begin{proof}
Assume that $\xi$ is a Busemann point.
By Proposition~\ref{prop:horo_thompson}, if $\xi$ is not of the form $r_x$
or $f_{C,x}$, with $x\in D$, then it is of the form $[r_x, f, c]$, with
$x\in\partial D$, and $f\in A(x)$, and $c\in\Raug$. Write $f= g\wedge g'$,
where $g$ and $g'$ are real-valued functions on $C$ that are $1$-Lipschitz
with respect to the Funk metric $\funk$.
Observe that this property of $g$ and of $g'$ implies that each of them
is the logarithm of an homogeneous function.
We consider the case when $c\le 0$; the other case is similar. We have
\begin{align*}
\xi &= \Big( (r_x+c) \vee (g \wedge g') \Big) \\
    &= \Big( (r_x+c) \vee g \Big) \wedge \Big( (r_x+c) \vee g' \Big).
\end{align*}
So, $\xi$ is the minimum of two functions on $C$ that are $1$-Lipschitz
with respect to Thompson's metric $\tmetric$. Since $\xi$ is a Busemann
point, it must be equal to one of them, say $(r_x+c) \vee g$.
From Lemma~\ref{lem:unique_split_on_cone}, we then get that $f=g$.
We have shown that if $f$ is written as the minimum of two functions that are
$1$-Lipschitz with respect to the Funk metric, then it equals one of them.
Since $f$ is a Funk horofunction, it follows that $f$ is a Busemann
point of the Funk geometry.

Functions of the form $r_x$ or $f_{C,x}$ with $x\in D$ are clearly Busemann
points of the Thompson geometry since they are the limits, respectively,
of the geodesics $t\mapsto e^t x$ and $t\mapsto e^{-t} x$.

Let $x\in\partial D$, and $f\in B(x)$, and $c\in\Raug$. Choose $\epsilon>0$.
We consider only the case where $c$ is finite; the case where it is infinite
is similar and easier.

Since $f$ is in $B(x)$, there exists,
by the proof of~\cite[Lemma~4.3]{walsh_hilbert},
a sequence $x_n$ in $D$ that converges to $f$ and to $r_x$, respectively,
in the Funk and reverse-Funk geometries, and furthermore is an
$\epsilon$-almost-geodesic with respect to both of these metrics.
So, as discussed above, by passing to a subsequence if necessary and
parametrising in the right way, we obtain an almost-geodesic converging
to $r_x$ in the reverse-Funk geometry.
Applying~\cite[Lemma~5.2]{walsh_stretch}, we get that
\begin{align}
\label{eqn:revLimit}
\rfunk(b,x_n) + r_x(x_n) \to 0.
\end{align}
In a similar fashion, one may show that
\begin{align}
\label{eqn:funkLimit}
\funk(b,x_n) + f(x_n) \to 0.
\end{align}
For each $n\in\N$, choose $z_n$ in the same projective class as $x_n$ such
that $\gauge{}{z_n}{b} = \gauge{}{b}{z_n}e^c$. So, for all $n\in\N$,
\begin{align*}
\thomp(b,z_n)
   &= \rfunk(b,z_n) \vee \funk(b,z_n) \\
   &= \rfunk(b,z_n) - c^- \\
   &= \funk(b,z_n) + c^+.
\end{align*}
Observe that both~(\ref{eqn:revLimit}) and~(\ref{eqn:funkLimit}) also hold
with $z_n$ in place of $x_n$ since, for each $n\in\N$, $x_n$ and $z_n$
are related by a positive
scalar. Combining all this, we have
\begin{align*}
\thomp(b,z_n) + [r_x, f, c](z_n)
   &= \big( \thomp(b,z_n) + r_x(z_n) + c^- \big)
           \vee \big( \thomp(b,z_n) + f(z_n) - c^+ \big) \\
   &= \big( \rfunk(b,z_n) + r_x(z_n) \big)
           \vee \big( \funk(b,z_n) + f(z_n) \big) \\
   &\to 0,
\end{align*}
as $n$ tends to infinity.

We have seen in the proof of Proposition~\ref{prop:horo_thompson}
that $z_n$ converges to $\xi:=[r_x, f, c]$ in the Thompson geometry.
We deduce that $H(\xi,\xi)=0$, and so $[r_x, f, c]$ is a Busemann point.
\end{proof}

Recall that we have extended the definition of $H$ by setting
$H(\xi+u,\eta +v):=H(\xi,\eta)+v-u$ for all Busemann points $\xi$ and $\eta$,
and $u,v\in[-\infty,0]$. We are also using the convention that $-\infty$
is absorbing for addition.

\begin{proposition}
\label{prop:detour_thompson}
The detour distance between two Busemann points $\xi$ and $\eta$ in the
horofunction boundary of the Thompson metric is $\delta(\xi,\eta)=\hil(x,y)$
if $\xi=r_x$ and $\eta=r_y$, with $x,y\in D$.
The same formula holds when $\xi=f_{C,x}$ and $\eta=f_{C,y}$,
with $x,y\in D$.
If $\xi=[r_x,f,c]$ and $\eta=[r_{x'},f',c']$, with $x,x'\in \partial D$,
$f\in B(x)$, $f'\in B(x')$, and $c,c'\in \Raug$, then
\begin{align*}
\delta(\xi,\eta)=
   \max\Big( H(\bar r_x,\bar r_{x'}), H(\bar f, \bar f') \Big)
   + \max\Big( H(\bar r_{x'},\bar r_{x}), H(\bar f', \bar f) \Big),
\end{align*}
where
\begin{align*}
\bar r_x &:= r_x + c^- , &\bar f &:= f - c^+ , \\
\bar r_{x'} &:= r_{x'} + c'^- , &\bar f' &:= f' - c'^+ .
\end{align*}
In all other cases, $\delta(\xi,\eta)=\infty$.
\end{proposition}

\begin{proof}
For $x$ and $y$ in $D$, we have, by Proposition~\ref{prop:detour},
\begin{align*}
H(r_x, r_y)
   &= \sup_{z\in C} \big( r_y(z) - r_x(z) \big) \\
   &= \sup_{z\in C} \big(
         \rfunk(z,y) - \rfunk(b,y) - \rfunk(z,x) + \rfunk(b,x)
         \big) \\
   &= \rfunk(x,y) - \rfunk(b,y) + \rfunk(b,x).
\end{align*}
Here we have used the triangle inequality to get an upper bound on the
supremum, and taken $z=x$ to get a lower bound.
Symmetrising, we get that $\detourmet(r_x,r_y)=\hil(x,y)$.

We use similar reasoning in the case where the two Busemann points are
of the form $f_{C,x}$ and $f_{C,y}$, with $x,y\in D$.

Now let $\xi:=[r_x, f, c]$ and $\eta:=[r_{x'}, f', c']$,
with $x,x'\in \partial D$,
$f\in B(x)$, $f'\in B(x')$, and $c,c'\in \Raug$.
By Proposition~\ref{prop:detour},
\begin{align*}
r_{x'}(\cdot) &\le r_x(\cdot) + H(r_{x}, r_{x'})
\qquad\text{and} \\
f'(\cdot) &\le f(\cdot) + H(f, f').
\end{align*}
If either $H(\bar r_x, \bar r_{x'})$ or $H(\bar f, \bar f')$ equals $+\infty$,
then $H(\xi,\eta)$ is trivially less than or equal to the maximum of the two.
So, assume that both quantities are less than $+\infty$. This rules out the
possibility that $c'>c=-\infty$, and the possibility that $c'<c=+\infty$. So,
\begin{align*}
\bar r_{x'}(\cdot) &\le \bar r_x(\cdot) + H(\bar r_{x}, \bar r_{x'})
\qquad\text{and} \\
\bar f'(\cdot) &\le \bar f(\cdot) + H(\bar f, \bar f').
\end{align*}

Therefore, 
\begin{align}
H(\xi,\eta)
   &= \sup_{z\in C} \Big( [r_{x'}, f', c'](z) - [r_{x}, f, c](z)
         \Big) \nonumber \\
   &= \sup_{z\in C} \Big(
         \big( \bar r_{x'}(z) \vee \bar f'(z) \big)
             - \big( \bar r_{x}(z) \vee \bar f(z) \big)
         \Big) \label{eqn:Hexpression} \\
   &\le \sup_{z\in C} \Big(
         \big( \bar r_{x}(z) + H(\bar r_{x}, \bar r_{x'})
            \vee \bar f(z) + H(\bar f, \bar f') \big)
             - \big( \bar r_{x}(z) \vee \bar f(z) \big)
         \Big) \nonumber \\
   &\le H(\bar r_{x}, \bar r_{x'}) \vee H(\bar f, \bar f'). \nonumber
\end{align}

We now wish to show that $H(\xi,\eta) \ge H(\bar r_{x}, \bar r_{x'})$.
This is trivial if $H(\bar r_{x}, \bar r_{x'})=-\infty$, so we assume
the contrary, which is equivalent to assuming that $\bar r_{x'}$ is finite
everywhere. Let $z\in C$. From~(\ref{eqn:Hexpression}), we have,
for all $\lambda>0$,
\begin{align*}
H(\xi,\eta)
   &\ge \bar r_{x'}(\lambda z)
             - \big( \bar r_{x}(\lambda z) \vee \bar f(\lambda z) \big) \\
   &= (\bar r_{x'} - \bar r_{x})(\lambda z)
          \wedge (\bar r_{x'} - \bar f)(\lambda z).
\end{align*}
Recall that $\bar r_{x}(\lambda z) = \bar r_{x}(z) - \log\lambda$
and $\bar f(\lambda z) = \bar f(z) + \log\lambda$, for all $\lambda>0$.
So, $(\bar r_{x'} - \bar r_{x})(\lambda z)$ is independent of $\lambda$.
Moreover, by choosing $\lambda$ small enough, we may make
$(\bar r_{x'} - \bar f)(\lambda z)$ as large as we wish.
We conclude that $H(\xi,\eta)\ge (\bar r_{x'} - \bar r_{x})(z)$.
Taking the supremum over $z\in C$ gives us what we wish.

The proof that $H(\xi,\eta)\ge H(\bar f, \bar f')$ is similar.

The result now follows on symmetrising.
\end{proof}

\begin{corollary}
\label{cor:singletons_thompson}
The singletons of the Thompson geometry are the Busemann points of the
form $r_x$ with $x$ an extreme point of $\closure D$, or of the form
$\log(\dotprod{y}{\cdot}/\dotprod{y}{b})$,
with $y$ an extremal generator of $C^*$.
\end{corollary}

\begin{proof}

Since the Busemann points of the form $r_x$, with $x\in D$, all lie in the
same part, none of them are singletons. Similarly, no Busemann point of the
form $f_{C,x}$, with $x\in D$, is a singleton.

Consider now a Busemann point $\xi:=[r_x,f,c]$,
with $x\in\partial D$, and $f\in B(x)$, and $c\in\Raug$.

If $c$ is finite, we define the Busemann point
$\xi^\epsilon := [r_x,f,c+\epsilon]$, for each $\epsilon\in\R$.
A simple calculation using Proposition~\ref{prop:detour_thompson} then gives
$\delta(\xi,\xi^\epsilon) = |\epsilon|$, for all
$\epsilon\in\R$. So $\xi^\epsilon$ is distinct from $\xi$,
but in the same part as it, for all $\epsilon\in\R\backslash\{0\}$,
and hence $\xi$ is not a singleton.

In the case where $c=\infty$ and $x$ is not an extreme point of $\closure D$,
we have that $x$ is in the relative interior of some extreme set of
$\closure D$ that contains another point $x'$ distinct from $x$.
One can then show using Proposition~\ref{prop:detour_thompson}
that $\xi':=[r_{x'},f',\infty]$ is distinct from $\xi$ but in the same
part, where $f'$ is any function.

The case where $c=-\infty$ and $f$ is not a singleton of the Funk geometry
is similar.

Now let $x$ be an extreme point of $\closure D$, so that $r_x$ is a singleton
of the reverse-Funk geometry.
Write $\xi:=r_x=[r_x,f,\infty]$, for any Funk horofunction $f$ in $B(x)$.
Let $\xi':=[r_{x'},f',c]$ be a Busemann point of the Thompson geometry
lying in the same part of the horofunction boundary as $\xi$.
Here, of course, $x'\in\partial D$, and $f'\in B(x')$, and $c\in\Raug$.
Since $\delta(\xi,\xi')$ is finite,
looking at Proposition~\ref{prop:detour_thompson}, we see that
$H(f,f')+\infty-c^+ < \infty$, and hence $c=\infty$.
Using this and the same proposition again, we get
$\infty>\delta(\xi,\xi')= \delta(r_x,r_{x'})$.
Since $r_x$ was assumed to be a singleton, we conclude that
$x'=x$, and therefore $\xi'=\xi$.
We have thus shown that $\xi$ is a singleton.

The proof is similar in the case of Busemann points of the form
$\log(\dotprod{y}{\cdot}/\dotprod{y}{b})$,
with $y$ an extremal generator of $C^*$.
\end{proof}

\section{Isometries of the Thompson metric}
\label{sec:thompson_isometries}

\newcommand\lin{\operatorname{lin}}
\newcommand\projection{P}

Let $C_1$, $C_2$, and $C$ be non-empty convex cones in the linear space $\linspace$.
We say that $C$ is the direct product of $C_1$ and $C_2$ if $C=C_1+C_2$
and $\lin C_1 \intersection \lin C_2 = \{0\}$.
Here $\lin$ denotes the linear span of a set.
In this case we write $C=C_1\directproduct C_2$.
If $C=C_1\directproduct C_2$, then $\lin C$
is the (linear space) direct sum of $\lin C_1$ and $\lin C_2$. Denoting by
$\projection_1$ and $\projection_2$ the corresponding projections, we have
$\projection_1(C)= C_1$ and $\projection_2(C)= C_2$.
We note that $C$ is relatively open if and only if both $C_1$ and $C_2$ are.

Let $C= C_1 \directproduct C_2$ be a product cone.
We write the Thompson metrics on $C_1$ and $C_2$ as $\thomp^1$
and $\thomp^2$, respectively.
Let $x_1,y_1\in C_1$ and $x_2,y_2\in C_2$.
One may easily verify that $x_1 + x_2 \le_C y_1 + y_2$ if and only if both
$x_1 \le_{C_1} y_1$ and $x_2 \le_{C_2} y_2$.
It follows that
\begin{align*}
\gauge{C}{x_1 + x_2}{y_1 + y_2}
    = \gauge{C_1}{x_1}{y_1} \vee \gauge{C_2}{x_2}{y_2},
\end{align*}
and hence that the Thompson metric on $C$ is
\begin{align}
\label{eqn:product_thompson_metric}
\thomp(x_1 + x_2, y_1 + y_2) = \thomp^1(x_1,y_1) \vee \thomp^2(x_2,y_2).
\end{align}

Assume that $C$ is proper, open, and convex.
Suppose that $C_2$ admits a gauge-reversing
bijection $\phi_2$, for example, one may think of $(0,\infty)$ with the
map $x\mapsto 1/x$, for $x\in(0,\infty)$.
Then, as pointed out in~\cite[Proposition 10.1]{noll_schaffer_orders_gauge},
there exists a Thompson metric isometry of $C$ that is neither
gauge-preserving nor gauge-reversing, namely the map $\phi\colon C\to C$
defined by
\begin{align*}
\phi(x_1+x_2) := x_1 + \phi_2(x_2),
\qquad\text{for all $x_1\in C_1$ and $x_2\in C_2$}.
\end{align*}
Indeed, we are applying here the identity map to the first component and
$\phi_2$ to the second. These maps are Thompson isometries on $C_1$ and $C_2$,
respectively, and so~(\ref{eqn:product_thompson_metric}) gives that $\phi$
is an isometry on $C$. However, $\phi$ is clearly neither homogeneous
nor anti-homogeneous, which implies by
Propositions~\ref{prop:gauge_preserving} and~\ref{prop:gauge_reversing}
that $\phi$ is neither gauge-preserving nor gauge-reversing.

We see from the following theorem that this is the only way in which such
isometries arise.

\begin{mytheorem}
Let $C$ and $C'$ be proper open convex cones, and
let $\phi\colon C\to C'$ be a surjective isometry of the Thompson metric.
Then, there exist decompositions $C=C_1\directproduct C_2$ and
$C'=C'_1\directproduct C'_2$ such that $\phi$ takes the form
$\phi(x_1+x_2)= (\phi_1(x_1)+\phi_2(x_2))$, where $\phi_1$ is a
gauge-preserving map from $C_1$ to $C'_1$, and $\phi_2$ is a gauge-reversing
map from $C_2$ to $C'_2$.
\end{mytheorem}

We will prove this theorem by considering the action of $\phi$ on the
horofunction boundary, more specifically, its action on the singletons.
Choose base-points $b$ and $b'$ in $C$ and $C'$, respectively, such that
$\phi(b)=b'$.

Recall that a singleton is a Busemann point that lies in a part consisting
of a single point, in other words, that is an infinite distance to every
other Busemann point in the detour metric.
We have seen in Corollary~\ref{cor:singletons_thompson}
that each singleton of the Thompson geometry is either a singleton of the
Funk geometry or a singleton of the reverse-Funk geometry, on the same cone.

Let $S$ be the set of functions of the form $\exp\after g$, where $g$ is a
singleton of the Thompson geometry on $C$. So each element of $S$ is the
restriction to $C$ of a function either of the form
$\dotprod{y}{\cdot}/\dotprod{y}{b}$, with $y$ an extremal generator of
$C^*$, or of the form $\gauge{}{x}{\cdot}/\gauge{}{x}{b}$, with
$x$ an extremal generator of $C$.
Denote by $F$ those of the former kind, and by $R$ those of the latter.
On $C'$, we define the sets of functions $S'$, $F'$, and $R'$ in the same way.

Since the action of $\phi$ on the horofunction boundary preserves the
detour metric, it maps parts to parts, and hence maps singletons to
singletons. Therefore, a function $f$ is in $S$ if and only if
$f\after\phi^{-1}$, which is its image under $\phi$,  is in $S'$.

Observe that each element of $F$ is a linear functional on $\linspace$, and together
they span the dual space $\linspace^*$. The idea of the proof is to examine which
of these linear functionals get mapped by $\phi$ to a linear functional,
and which get mapped to something non-linear.

Let $F_1$ denote the elements of $F$ that are mapped to elements of $F'$,
and $F_2$ denote those that are mapped to elements of $R'$. Similarly,
define an element $f'$ of $F'$ to be in either $F_1'$ or $F_2'$ depending
on whether its image $f'\after \phi$ under $\phi^{-1}$ is in $F$ or $R$.
So we have the picture given in Figure~\ref{fig:singleton_action}.

\begin{figure}
\input{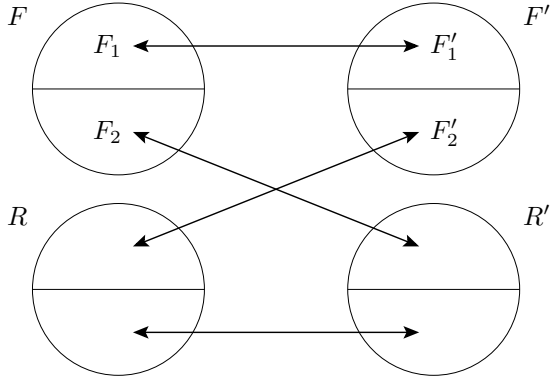}
\caption{The action of an isometry on the singletons in the boundary.}
\label{fig:singleton_action}
\end{figure}

Define
\begin{align*}
C_1 &:= \relint \Big\{
           z\in\closure C \mid \text{$f(z)=0$ for all $f\in F_2$}
       \Big\}
\qquad\text{and} \\
C_2 &:= \relint \Big\{
           z\in\closure C \mid \text{$f(z)=0$ for all $f\in F_1$}
       \Big\}.
\end{align*}
Observe that $\closure C_1$ and $\closure C_2$ are exposed faces of
$\closure C$. We define $C'_1$ and $C'_2$ in an analogous way.

\begin{lemma}
\label{lem:funks_determine}
If $x$ and $y$ are in $C_1$, and $f(x)=f(y)$ for all $f\in F_1$, then $x=y$.
Similarly, if $x$ and $y$ are in $C_2$, and $f(x)=f(y)$ for all $f\in F_2$,
then $x=y$.
\end{lemma}
\begin{proof}
Under the assumptions of the first statement,
$f(x)=f(y)$ for all $f\in F_1\union F_2$.
However, $F_1\union F_2$ is exactly the set of extreme points $f$ of
$\closure C^*$ satisfying $f(b)=1$. Since this set spans $\linspace^*$, we have $x=y$.

The proof of the second statement is similar.
\end{proof}

\begin{lemma}
\label{lem:unique_decomposition}
Suppose $z\in C$ can be written $z = x_1 + x_2 = y_1 + y_2$, with $x_1$ and
$y_1$ in $C_1$, and $x_2$ and $y_2$ in $C_2$.
Then $x_1 = y_1$ and $x_2 = y_2$.
\end{lemma}
\begin{proof}
Each $f\in F_2$ is linear and takes the value zero on $C_1$,
and so $f(z) = f(x_2) = f(y_2)$. Therefore, by Lemma~\ref{lem:funks_determine},
$x_2=y_2$.

The proof that $x_1 = y_1$ is similar.
\end{proof}

Define
\begin{align*}
\proj_1(z,\alpha)
   &:= \frac{1}{\alpha} \phi^{-1}\big(\alpha\phi(z)\big)
\qquad\text{and} \\
\proj_2(z,\alpha)
   &:= \frac{1}{\alpha} \phi^{-1}\Big(\frac{1}{\alpha}\phi(z)\Big),
\end{align*}
for all $z\in C$ and $\alpha\in(0,\infty)$.

\begin{lemma}
\label{lem:projection_maps}
The cone $C$ is the direct product of $C_1$ and $C_2$, and the maps
$\proj_1(z) := \lim_{\alpha\to\infty} \proj_1(z,\alpha)$ and
$\proj_2(z) := \lim_{\alpha\to\infty} \proj_2(z,\alpha)$
are the projection maps onto $C_1$ and $C_2$, respectively.
\end{lemma}
\begin{proof}
Let $z\in C$. For $f\in F_1$, we have that $f\after\phi^{-1}$ is the
exponential of a Funk horofunction, and is therefore homogeneous.
In this case,
\begin{align*}
f\big(\proj_1(z,\alpha)\big)
   &= \frac{1}{\alpha} f\after\phi^{-1}\big(\alpha\phi(z)\big) \\
   &= f(z).
\end{align*}
On the other hand, for $f\in F_2$, we have that $f\after\phi^{-1}$ is
the exponential of a reverse-Funk horofunction, and is therefore
anti-homogeneous, which gives that
$f\big(\proj_1(z,\alpha)\big) = f(z)/\alpha^2$.

So, the limit of $f\big(\proj_1(z,\alpha)\big)$ as $\alpha$ tends to infinity
exists for all $f\in F$. This implies that the limit $\proj_1(z)$ defined
in the statement of the lemma exists.

Moreover, we have
\begin{align}
\label{eqn:projection_coords}
f\big(\proj_1(z)\big) =
   \begin{cases}
   f(z), & \text{for $f\in F_1$}; \\
   0,    & \text{for $f\in F_2$}.
   \end{cases}
\end{align}
So, $\proj_1(z)$ is in $\closure C_1$.
Similarly, one can show that the limit $\proj_2(z)$ exists and
lies in $\closure C_2$, and that
\begin{align}
\label{eqn:projection_coords_b}
f\big(\proj_2(z)\big) =
   \begin{cases}
   0,    & \text{for $f\in F_1$}; \\
   f(z) & \text{for $f\in F_2$}.
   \end{cases}
\end{align}
Observe that $f\big(\proj_1(z)\big)+f\big(\proj_2(z)\big) = f(z)$,
for all $f\in F$. It follows that $z=\proj_1(z)+\proj_2(z)$.
We have shown that $C$ is a subset of $\closure C_1 + \closure C_2$.
It follows that $\closure C$ is also a subset of this set
since the latter set is closed. That $\closure C_1 + \closure C_2$ is a subset
of $\closure C$ follows from the fact that $\closure C_1$ and $\closure C_2$
are contained in $\closure C$.

Lemma~\ref{lem:unique_decomposition} implies that
\begin{align*}
\lin \closure C_1 \intersection \lin \closure C_2
   = \lin C_1 \intersection \lin C_2
   = \emptyset.
\end{align*}
So we see that the cone $\closure C$ is the direct product of
$\closure C_1$ and $\closure C_2$.
It follows immediately that $C$ is the direct product of $C_1$ and $C_2$.
\end{proof}

We define maps $\proj_1'$ and $\proj_2'$ on $C'$ analogously to how we
defined $\proj_1$ and $\proj_2$.

\begin{lemma}
\label{lem:map_first_coord}
Let $x$ and $y$ in $C$ be such that $\proj_1(x)=\proj_1(y)$.
Then $\proj'_1\big(\phi(x)\big)=\proj'_1\big(\phi(y)\big)$.
\end{lemma}
\begin{proof}
From~(\ref{eqn:projection_coords}), we have $f(x)=f(y)$, for all $f\in F_1$.
Equivalently,
$f\after\phi^{-1}(\phi(x)) = f\after\phi^{-1}(\phi(y))$,
for all $f\in F_1$.
But $f\after\phi^{-1}$ is in $F'_1$ if and only if $f$ is in $F_1$,
and so $f'(\phi(x)) = f'(\phi(y))$ for all $f'\in F'_1$.
It follows that $\proj'_1(\phi(x)) = \proj'_1(\phi(y))$.
\end{proof}

We say that a function $f$ on a product cone $C_1\directproduct C_2$
is independent of the first component if $f(x)=f(y)$
whenever $\proj_2(x)=\proj_2(y)$.
Similarly, we say that $f$ is independent of the second component
if $f(x)=f(y)$ whenever $\proj_1(x)=\proj_1(y)$.

Equations~(\ref{eqn:projection_coords}) and (\ref{eqn:projection_coords_b})
imply, respectively, that each function in $F_1$ is independent of the
second component, and each function in $F_2$ is independent of the 
first component.

Let $\thomp^1$ and $\thomp^2$  be the Thompson metrics on $C_1$ and $C_2$,
respectively. Since $C$ is the direct product of $C_1$ and $C_2$,
we may write
\begin{align*}
\thomp(x_1+x_2, y_1+y_2) = \thomp^1(x_1,y_1) \vee \thomp^2(x_2,y_2),
\end{align*}
for all $x_1,y_1\in C_1$ and $x_2,y_2\in C_2$.
So, $(C,\thomp)$ is the $\ell_\infty$-product of the spaces
$(C_1,\thomp^1)$ and $(C_2,\thomp^2)$ in the sense of
Section~\ref{sec:product}.
It was shown there, in Corollary~\ref{cor:product_singletons},
that each singleton of such a product depends only on one of the two
components. We conclude that each element of $S$ is either independent of
the first component or independent of the second.

\begin{lemma}
\label{lem:independent}
Let $f\in S$. Then, $f$ is independent of the first component if and only if
$f\after\phi^{-1}$ is.
Likewise, $f$ is independent of the second component if and only if
$f\after\phi^{-1}$ is.
\end{lemma}
\begin{proof}
Assume $f\after\phi^{-1}$ is independent of the second component.
Let $x$ and $y$ in $C$ be such that $\proj_1(x)=\proj_1(y)$.
So, by Lemma~\ref{lem:map_first_coord},
$\proj'_1\big(\phi(x)\big) = \proj'_1\big(\phi(y)\big)$. Therefore,
$f\after\phi^{-1}\big(\phi(x)\big) = f\after\phi^{-1}\big(\phi(y)\big)$,
or equivalently, $f(x)=f(y)$. We conclude that $f$ is independent of the
second component.

The implication in the opposite direction is proved in a similar manner.

Now assume that $f\after\phi^{-1}$ is independent of the first component,
and so not independent of the second.
Since $f$ is in $S$, it must be independent of either the first or
second component. However, the latter possibility is ruled out by what we
have just proved. Again, the implication in the opposite direction is similar.
\end{proof}

\begin{lemma}
\label{lem:map_second_coord}
Let $x$ and $y$ in $C$ be such that $\proj_2(x)=\proj_2(y)$.
Then $\proj'_2\big(\phi(x)\big)=\proj'_2\big(\phi(y)\big)$.
\end{lemma}
\begin{proof}
Let $f'\in F'_2$. So, $f'\after\phi$ is in $R$, and,
by Lemma~\ref{lem:independent},
it is independent of the first component since $f'$ is. In particular,
$f'(\phi(x)) = f'(\phi(y))$. Using~(\ref{eqn:projection_coords_b})
we get that $f'\big(\proj'_2(\phi x)\big)=f'\big(\proj'_2(\phi y)\big)$,
for all $f'\in  F'_2$.
But the same equation also holds for all $f'\in  F'_1$,
since, by~(\ref{eqn:projection_coords_b}), both sides are zero in this case.
The conclusion follows, since the set of linear functions $F'_1\union F'_2$
spans the dual space of $\linspace':=\lin C'$.
\end{proof}

\begin{proof}[Proof of Theorem~\ref{thm:thompson_isometries}]
It was shown in Lemma~\ref{lem:projection_maps} that $C$ and $C'$
decompose in the way claimed, and in Lemmas~\ref{lem:map_first_coord}
and~\ref{lem:map_second_coord} that $\phi$ is of the form
$\phi(x_1 + x_2) = \phi_1(x_1) + \phi_2(x_2)$, for all $x_1\in C_1$
and $x_2\in C_2$,
for some maps $\phi_1\colon C_1\to C'_1$ and $\phi_2\colon C_2\to C'_2$.

Since $C$ is a direct product of $C_1$ and $C_2$, its Thompson metric
can be written
\begin{align*}
\thomp(x_1 + x_2, y_1 + y_2)
   = \max\Big(\thomp^1(x_1,y_1),\thomp^2(x_2,y_2)\Big),
\end{align*}
in terms of the Thompson metrics on $C_1$ and $C_2$.
A similar expression holds for~$\thomp'$.
So, for $z\in C_1$ and $x_2, y_2\in C_2$, we have
\begin{align*}
\thomp(z + x_2, z + y_2) &= \thomp^2(x_2,y_2)
\end{align*}
and
\begin{align*}
\thomp'\big(\phi(z + x_2), \phi(z + y_2)\big)
   &= \thomp'\big(\phi_1(z) + \phi_2(x_2), \phi_1(z) + \phi_2(y_2)\big) \\
   &= \thomp^{\prime 2}\big(\phi_2(x_2),\phi_2(y_2)\big).
\end{align*}
We conclude that $\phi_2$ is an isometry from $(C_2,\thomp^2)$ to
$(C'_2,\thomp^{\prime 2})$.

Moreover, for all $x\in C$ and $\lambda>0$,
\begin{align*}
\proj'_2\big(\phi(\lambda x)\big)
   &= \lim_{\alpha\to\infty}
         \frac{1}{\lambda}
         \frac{\lambda}{\alpha}
         \phi\Big(\frac{\lambda}{\alpha} x\Big) \\
   &= \frac{1}{\lambda} \proj'_2\big(\phi(x)\big).
\end{align*}
Hence, $\phi_2$ is anti-homogeneous.

We now apply Proposition~\ref{prop:gauge_reversing} to get that $\phi_2$
is gauge-reversing.

A similar argument shows that $\phi_1$ is also a Thompson-metric isometry,
but this time homogeneous, and hence gauge-preserving
by Proposition~\ref{prop:gauge_preserving}.
\end{proof}



\bibliographystyle{plain}
\bibliography{gaugereversing}

\end{document}